\documentclass[11pt, twoside, reqno]{article} 
\usepackage[margin=1in]{geometry}        
\usepackage{graphicx}		
\usepackage[dvipsnames]{xcolor}
\usepackage{mathtools,amsthm,enumitem,overpic}

\usepackage[T1]{fontenc}
\usepackage{lmodern}

\usepackage{caption}
\usepackage[linktocpage]{hyperref}
\hypersetup{
    colorlinks=true,
    linkcolor=blue,
    hypertexnames=false,
    filecolor=black,      
    urlcolor=black,
    citecolor=red
}
\title{Klein--Maskit combination theorem for Anosov subgroups:\\ Free products}
\author{Subhadip Dey \and Michael Kapovich}

\usepackage{amssymb}

\usepackage[]{cleveref}
\theoremstyle{thmstyleone}
\newtheorem{theorem}{Theorem}[section]
\newtheorem{proposition}[theorem]{Proposition}
\newtheorem{lemma}[theorem]{Lemma}
\newtheorem{corollary}[theorem]{Corollary}
\newtheorem{maintheorem}{Theorem}

\theoremstyle{remark}
\newtheorem{remark}[theorem]{Remark}

\theoremstyle{definition}
\newtheorem{definition}[theorem]{Definition}

\newtheorem{case}{Case}
\newtheorem{example}[theorem]{Example}

\def\H{\mathbb{H}}
\def\N{\mathbf{N}}
\def\R{\mathbf{R}}
\def\Z{\mathbf{Z}}

\def\acts{\curvearrowright}
\def\ra{\rightarrow}
\def\Fu{{\rm Fu}}

\DeclareMathOperator{\PSL}{PSL}

\def\emod{{\eta_{\rm mod}}}
\def\Ft{{\rm Flag}({\tau_{\rm mod}})}
\def\Fs{{\rm Flag}({\sigma_{\rm mod}})}
\def\Pt{P_{\tau_{\rm mod}}}

\def\e{\varepsilon}
\def\G{\Gamma}
\def\g{\gamma}

\def\Lt{{\Lambda_{\tau_{\rm mod}}}}

\def\raf{\xrightarrow{\rm flag}}
\def\raC{\xrightarrow{\rm Cay}}
\def\smod{{\sigma_{\rm mod}}}

\def\tits{\partial_{\rm Tits}}
\def\tmod{{\tau_{\rm mod}}}

\def\vb{\partial_\infty}
\def\vi{\partial_{\rm I}\Gamma}
\def\vii{\partial_{\rm II}\Gamma}

\def\<{\langle}
\def\>{\rangle}

\DeclareMathOperator{\Int}{int}
\DeclareMathOperator{\isom}{Isom}
\DeclareMathOperator{\rell}{rl}

\DeclareMathOperator{\St}{St_{Fu}}
\DeclareMathOperator{\geo}{\partial_{\infty}}
\DeclareMathOperator{\id}{id}

\DeclarePairedDelimiterX{\inp}[2]{\langle}{\rangle}{#1, #2}

\begin{document}

\maketitle

\begin{abstract}
  We prove a generalization of the classical Klein--Maskit combination theorem, in the free product case, in the setting of Anosov subgroups.
 Namely, if $\G_A$ and $\G_B$ are Anosov subgroups of a semisimple Lie group $G$ of noncompact type, then under suitable topological assumptions, the group generated by $\G_A$ and $\G_B$ in $G$ is again Anosov, and is naturally isomorphic to the  free product $\G_A*\G_B$. 
 Such a generalization was conjectured in \cite{MR4002289}.
\end{abstract}

The classical Klein--Maskit combination theorem for Kleinian groups (discrete isometry groups of the hyperbolic space $\H^n$) establishes 
sufficient conditions for a subgroup $\G< G=\isom(\H^n)$ generated by two discrete subgroups $\Gamma_A, \Gamma_B $ of $G$ to be discrete and isomorphic to the free product $\Gamma_A* \Gamma_B $, see \cite{maskit1965klein,maskit1968klein,maskit1971klein,maskit1993klein} as well as Maskit's book \cite[Theorem C.2, p. 149]{maskit:book} 
and papers by Ivascu, \cite{Ivascu}, Li, Ohshika and Wang, \cite{LOW}. In this situation, the group $\G$ is said to be obtained via {\em Klein combination} of the subgroups $\Gamma_A, \Gamma_B$. Moreover, Maskit gave sufficient conditions for the Maskit combination of two geometrically finite subgroups of $\G$ to be again geometrically finite and described the limit set of $\G$ in terms of those of $\G_A, \G_B$. Maskit's condition was formulated in terms of {\em fundamental domains} of group actions on the boundary sphere $S^{n-1}$ of $\H^n$. Further generalizations of the Klein--Maskit combination theorem in the context of group actions on Gromov-hyperbolic spaces appear in \cite{baker2008combination, GITIK199965, pedroza2008combination, pedroza2009, MR2994828}.

In the last decade, {\em Anosov subgroups} of higher rank Lie groups have emerged as a higher-rank generalization of {\em convex cocompact} Kleinian groups. The goal of this article is to prove an analog of Maskit's theorem in the setting of {\em Anosov subgroups} of semisimple Lie groups $G$.\footnote{We note that Maskit also proved some far-reaching generalizations of the above combination theorem, dealing with the cases of amalgamated free products and HNN extensions. Such a generalization is not an objective of the present paper but is discussed in our subsequent work in \cite{DK23}.} 
 An earlier form of such an analogue was proven in our paper with Bernhard Leeb, \cite{MR4002289}, using the {\em local-to-global principle} for Morse quasigeodesics and extending the earlier work by Kapovich, Leeb and Porti \cite{morse}. However, our earlier work required passage to certain finite index subgroups in the given Anosov subgroups. In the paper \cite{MR4002289} we conjectured a sharp form of the combination theorem for Anosov subgroups, which exactly matches Maskit's conditions for Kleinian groups. The main goal of the present paper  is to prove this conjecture. Unlike \cite{MR4002289}, the methods of the present paper are purely dynamical, relying upon another characterization (given in the work of Kapovich, Leeb and Porti, \cite{MR3736790}) of Anosov subgroups, as {\em asymptotically embedded} subgroups. The appropriate flag-manifolds $\Ft$ (quotients of $G$ by suitable parabolic subgroups $\Pt<G$) serve as substitutes to the ideal boundary sphere $S^{n-1}$ of $\H^n$ in Maskit's setting. The subsets $A, B\subset \Ft$ appearing below act as replacements to the complements of the interiors of fundamental domains of the subgroups $\G_A, \G_B$ in $S^{n-1}$ used by Maskit. The {\em antipodality} condition in the theorem is the appropriate higher-rank analogue of the disjointness condition imposed by Maskit. The main difficulty in the proof of our main theorem is to establish that certain sequences of {\em nested images} of the sets $A, B$ under {\em alternating} sequences in the free product $\G_A * \G_B$ have singleton intersections, see  \Cref{lem:int_singleton}. (This intersection property allows us to prove both $\tmod$-regularity of $\G$ and construct a continuous equivariant embedding of the Gromov boundary of $\G$ into $\Ft$.) The intersection property would have been easy if we knew that nontrivial elements of $\G_A$, resp. $\G_B$, {\em uniformly contract} on $B$, resp. $A$. Such uniform contraction, in general, fails under the assumption of the Klein--Maskit combination theorem even in the case of $\H^n$, see Example \ref{ex1}.

 \medskip
 Let $G$ be a noncompact real semisimple Lie group with a finite center, 
 $X = G/K$ be the associated symmetric space of noncompact type, $\smod$ be a model spherical chamber (a model facet) in the Tits building $\tits X$ of $X$, $\iota : \smod \ra \smod$ be the {\em opposition involution}, and $\tmod\subset \smod$ be an $\iota$-invariant face.
 We will assume some mild conditions on $G$, 
 see \Cref{sec:basics} for more details.
 
 Our main result is:
    
\begin{maintheorem}[Combination Theorem]\label{thm:main}
 Suppose that $A, B\subset \Ft$
 are (disjoint) compact sets with nonempty interiors which are {\em antipodal} (see \Cref{def:antipodal}) to each other. 
 Let $\G_A$ and $\G_B$ be $\tmod$-Anosov subgroups\footnote{Although we do not state original definition of $\tmod$-Anosov subgroups (due to Labourie \cite{MR2221137}) in this paper, an equivalent characterization of these as $\tmod$-asymptotically embedded subgroups is discussed in \Cref{sec:discrete_groups}.
The proof of our main result relies upon this characterization.} 
of $G$ such that  
 all nontrivial elements $\alpha\in \G_A,\, \beta\in \G_B$ 
 satisfy
 $$
 \alpha(B)\subset \Int(A), \quad \beta(A)\subset \Int(B). 
 $$
 Then:
 \begin{enumerate}[label=(\roman*)]
 \itemsep0em
 \item The subgroup $\G< G$ generated by $\G_A, \G_B$ is naturally isomorphic to the abstract free product $\G_A * \G_B$. 
 
 \item The subgroup $\G< G$ is $\tmod$-Anosov. 
 
 \item The {\em $\tmod$-limit set} of $\G$ (see \Cref{sec:discrete_groups}) is contained in $A\cup B$.
 \end{enumerate}
\end{maintheorem}

This theorem has an immediate generalization to the case of several Anosov subgroups of $G$ that we will discuss in  \Cref{sec:consequences}.

To end this introduction, we would like to mention that Danciger, Guéritaud, and Kassel in \cite{DGK17} announced a combination theorem for the free products of {\em convex cocompact} subgroups of 
$G = {\rm PGL}(n,\R)$. See Proposition 12.4 in their paper for a precise statement.
Notably, by Theorem 1.15 of their paper, the class of word-hyperbolic convex cocompact subgroups of $G$ coincides with the class of {\em projective} Anosov subgroups of $G$ preserving some properly convex domain in ${\mathbb P}(\R^n)$.
Therefore, the combination theorem announced in their paper is also closely related to our main result.

 \subsubsection*{Organization of the paper} 

\smallskip\noindent
In \Cref{sec:prelim}, we prove some preliminary results needed for the proof of \Cref{thm:main}.
A crucial result proven in this section is that the {\em stars in flag-manifolds are connected}, see \Cref{sec:flag_var}, which is also of independent interest.
In \Cref{sec:discreteness} and \Cref{sec:regularity}, we show that under a suitable condition, a pair of $\tmod$-regular subgroups of $G$ generate another one; see \Cref{thm:regularity} for the precise statement.
In \Cref{sec:boundary}, we show that, under the hypothesis of \Cref{thm:main}, the subgroup $\G_A * \G_B< G$ is {\em $\tmod$-boundary embedded}.
In \Cref{sec:mainproof}, we conclude the proof of \Cref{thm:main}.
In the final section, \Cref{sec:consequences}, we discuss some consequences.

 \subsubsection*{Acknowledgement} 
 We are grateful to our anonymous referee for carefully reading this paper and making several useful comments.

\section{Preliminary notions and results}\label{sec:prelim}

The goal of this section is to set up our notation and terminology, and prove some preliminary results which will be used in the later sections.

\subsection{Reduced forms}\label{sec:reduced}
 Let $\G_A$ and $\G_B$ be two groups and let $\G = \G_A*\G_B$ be their  free product.
 We will regard $\G_A$ and $\G_B$ as subgroups of $\G$ under their natural embeddings in $\G$.
 Each element $\g\in \G$ can be written as a unique word of the form
\begin{equation}\label{eqn:reduced}
 \g = \g_k \cdots \g_1,
\end{equation}
where $k\in \N$ depends on $\g$,
 such that the following conditions are satisfied:
\begin{enumerate}[label=(\roman*)]\itemsep0em
 \item Each letter $\g_i$ in the above expression belong to either $\G_A$ or $\G_B$.
 \item No two successive letters $\g_{i+1}, \g_i$ in the above expression belong to the same group $\G_A$ or $\G_B$.
\end{enumerate}
The unique expression \eqref{eqn:reduced} for $\g$ is called the {\em reduced form} of $\g$, and the number of nontrivial letters in the right side of \eqref{eqn:reduced} is called the {\em relative (word-)length} of $\g$ and denoted  by 
$
 \rell(\g).
$

An element $\g'\in\G$ is called a {\em leftmost} (resp. {\em rightmost}) {\em subword} of $\g\in\G$ if the reduced form of $\g'$ is obtained by deleting some letters from the right (resp. left) of the reduced form of $\g$.
For example, if $\g\in\G$ has the reduced form $\g = \alpha\beta\delta$, then the elements $\alpha$, $\alpha\beta$, and $\alpha\beta \delta$ of $\G$ are leftmost subwords of $\g$ whereas the elements $\delta$, $\beta\delta$, and $\alpha\beta \delta$ are rightmost subwords of $\g$.
 
 \subsection{Hausdorff distance and convergence}\label{sec:convergence}
 
Given a compact metric space $(Z,d)$, we have the Hausdorff distance on the  set ${\mathcal C}(Z)$ of nonempty closed subsets of $Z$. This distance defines the {\em topology of Hausdorff-convergence} on ${\mathcal C}(Z)$, which coincides with the {\em Chabauty topology} on 
${\mathcal C}(Z)$. 
Under this topology, a sequence  $(A_n)$ in ${\mathcal C}(Z)$ {\em converges} to a singleton $\{z\}\in {\mathcal C}(Z)$, denoted by
\[
 A_n \to z,
\]
if and only if the diameter of $A_n\cup \{z\}$ goes to zero as $n\to \infty$.
 

\subsection{Geometric background}\label{sec:basics}

Let $G$ be a noncompact real semisimple Lie group with a finite center.
We will impose some additional assumptions on $G$ given at the end of this subsection.
Let $K<G$ be a maximal compact subgroup of $G$. 
The {\em symmetric space} of $G$ is the simply-connected space $X = G/K$, equipped with a $G$-invariant Riemannian metric.
It is a standard fact that such a Riemannian metric has {\em non-positive curvature}, and moreover, $X$ has no flat de Rham factors.
We refer to Eberlein's book, \cite{eberlein}, for a detailed discussion of symmetric spaces.

The ideal boundary of $X$ is the set of asymptotic classes of geodesic rays in $X$.
The ideal boundary carries a natural $G$-invariant {spherical building} structure, called the {\em Tits building} of $X$, and is denoted by $\tits X$: The apartments in this building are the ideal boundaries of the maximal flats in $X$. Top-dimensional simplices ({\em facets}) in $\tits X$ are called (spherical) {\em chambers} and codimension one simplices are called {\em panels}. The group $G$ acts transitively on the sets of apartments and chambers of $\tits X$. Let $F_{\rm mod} \subset X$ denote  a chosen maximal flat; the ideal boundary of this flat, denoted by $a_{\rm mod}$, is called the {\em model apartment}.
The action of $G_{a_{\rm mod}}$, the stabilizer of $a_{\rm mod}$ in $G$, on $a_{\rm mod}$ factors through a finite group $W$, called the {\em Weyl group} associated to $G$.
A chosen fundamental domain for the action $W \acts a_{\rm mod}$ is called the {\em model chamber}, and is denoted by $\smod$.
The other facets in $\tits X$ are the $G$-translates of $\smod$.

The reflection about a(ny) point $x\in F_{\rm mod}$ determines an involution ${\rm inv} : a_{\rm mod} \ra a_{\rm mod}$ which preserves the chambers of $a_{\rm mod}$.
The {\em longest element} $w_0\in W$ is the unique element which sends the chamber ${\rm inv}(\smod) \subset a_{\rm mod}$ to $\smod$.
The composition $w_0 \circ {\rm inv} : \smod \ra \smod$ is a simplicial map, called the {\em opposition involution} and denoted by $\iota$.

The stabilizer of $\smod$ in $G$ is called the {\em minimal parabolic subgroup}, denoted by $P_\smod$.
More generally, the stabilizers of faces $\emod\subset\smod$ are the {\em parabolic subgroups} of $G$, denoted by $P_\emod$.
Therefore, the space of all simplices in $\tits X$ of {\em type}\footnote{A simplex $\eta\in \tits X$ is called of {\em type} $\emod$ if there exists $g\in G$ such that $g\emod = \eta$.} $\emod$ can be identified with the {\em partial flag variety}
\[
 {\rm Flag}(\emod) \coloneqq G/P_\emod.
\]
It is easy to see that $\eta'_{\rm mod} \subset \emod$ if and only if $P_\emod < P_{\eta'_{\rm mod}}$.
The {\em full flag variety}, ${\rm Flag}(\smod) = G/P_{\smod}$, is also known as the {\em Furstenberg boundary} of $X$.

For a face $\emod\subset \smod$, it is often convenient to use the notation $\pm\emod$ to denote the pair $\emod,~\iota\emod$, respectively.
A pair of points $\eta_\pm\in {\rm Flag}(\pm\emod)$ is called {\em antipodal}
if there exists a complete geodesic line in $X$, 
which is forward (resp. backward) asymptotic to an interior point of $\eta_+$ (resp. $\eta_-$).
Equivalently, antipodal simplices in $\tits X$ are those which are swapped by some Cartan involution of $X$.
To avoid possible confusions, we remark here that pairs of antipodal points in $\tits X$ are necessarily distinct.

For $\eta\in {\rm Flag}(\emod)$, let
\begin{equation}\label{def:cell}
  C(\eta) \coloneqq \{\eta_- \in {\rm Flag}(\iota\emod) \mid \text{$\eta_-$ is antipodal to $\eta$} \}.
\end{equation}
This is an open dense cell in ${\rm Flag}(\iota\emod)$.
The complement of $C(\eta)$ in ${\rm Flag}(\iota\emod)$, denoted by $E(\eta)$, is the {\em exceptional subvariety} for $\eta$.

In the following definition, we assume that $\tmod \subset \smod$ is an $\iota$-invariant face, i.e., $\tmod = \iota\tmod$.
 
\begin{definition}[Antipodality]\label{def:antipodal}
 A subset $\Lambda\subset {\rm Flag}(\tmod)$ is called {\em antipodal} if every pair of distinct points $\lambda_\pm\in \Lambda$ is antipodal.
If $A$ and $B$ are subsets of ${\rm Flag}(\tmod)$, then they are called {\em antipodal to each other} (or $A$ is {\em antipodal} to $B$) if for all points $a\in A$, $b\in B$, $a$ and $b$ are antipodal.
\end{definition}

Throughout, we  impose the following mild assumptions on $G$:
\begin{enumerate}[label=(\roman*)]\itemsep0em
 \item The group $G$ is commensurable with the full isometry group of $X$.
 \item The Tits building of $X$ is {\em thick}, i.e., every {\em panel}  in $\tits X$ is a face of at least three different chambers. 
\end{enumerate}
These are standing assumptions in the papers \cite{MR3811766,MR3736790,KLP:Morse}  by  Kapovich, Leeb, and Porti we rely upon in this work.
Moreover, we also refer to these papers for more details on the notions introduced in this section.

\subsection{Flag varieties}\label{sec:flag_var}
Let $\emod\subset\smod$ be a face. Consider the canonical $G$-equivariant algebraic morphism
\begin{equation}\label{eqn:proj}
 \pi_\emod: \Fs \ra {\rm Flag}(\emod).
\end{equation}
The fiber of $\pi_\emod$ over any point $\eta\in {\rm Flag}(\emod)$ is called the {\em star} of $\eta$, denoted by $\St(\eta)$, and is the smooth subvariety of $\Fs$ consisting of all chambers $\sigma$ which contain $\eta$ as a face. These fibers are diffeomorphic to the quotients $P_{\emod}/P_{\smod}$. 

\begin{lemma} Let $\emod$ and $\tmod$ be any faces of $\smod$. Then: 
\label{lem:connected fibers} 
\begin{enumerate}[label=(\roman*)]\itemsep0em
 \item Fibers of the map $\pi_{\emod}$ in \eqref{eqn:proj} are connected. 
\item The projection of $\St(\eta)$ to $\Ft$ is connected. 
\end{enumerate}
\end{lemma}
\begin{proof}
  (i) Any two chambers $\sigma_1, \sigma_2\in \St(\eta)$ lie in a common apartment $a$ in the spherical building $\tits X$. Then, there exists a gallery\footnote{A {\em gallery} is a finite sequence of chambers such that every two consecutive chambers in the sequence are adjacent, i.e., share a {\em panel} (a codimension one face).} in $a$ consisting of chambers in the star of $\eta$ and connecting $\sigma_1$ and $\sigma_2$. Hence, it suffices to prove that 
$\sigma_1, \sigma_2$ lie in the same component of  $\St(\eta)$, provided that $\sigma_1, \sigma_2$ share a panel $\tau$ and lie in a common spherical apartment $a\subset \tits X$. Let $w$ denote an element of the Weyl group $W_a$ of the apartment $a$ fixing $\tau$ 
(pointwise) and swapping the chambers $\sigma_1, \sigma_2$. (The element $w$ comes from an isometry $g_w\in G$ preserving the apartment $a$.) 

Let $\hat\tau\subset a$ be the unique simplex of type $\iota\tmod$ opposite to $\tau$. Let $f\subset X$ denote a flat of dimension one less than the rank of $X$ such that $\geo f\subset a$ and $\tau, \hat\tau$ are both contained in $\geo f$. The parallel set $P(f)\subset X$ of the flat $f$  splits isometrically as the product of a rank 1 symmetric subspace $Y\subset X$ and $f$ and the apartment $a$ is contained in the ideal boundary of $P(f)$. The isometry group of $P(f)$ fixes (pointwise) the ideal boundary of $f$ and, hence, the simplex $\tau$. We connect a generic point $\xi\in \tau$ to its antipode $\hat\xi\in \hat \tau$ by a geodesic $\gamma_1$ on $a$ passing through the interior of $\sigma_1$. Then the geodesic $w(\gamma_1)=\gamma_2$ also connects $\xi, \hat\xi$ and passes through the interior of $\sigma_2$. There exists a (unique) point $\zeta_1\in \gamma_1$ which lies in the ideal boundary of $\geo Y\cap a$.  Hence, $w(\zeta_1)=\zeta_2$ is in the ideal boundary of $Y$ as well. The connected component of the isometry group of $Y$ acts transitively on the ideal boundary of $Y$ (since $Y$ has rank 1). If follows that there is a 1-parameter family $h_t, t\in [1,2]$, of isometries of $X$ preserving $Y$ and fixing $\geo f$ pointwise 
such that $h_1=\id$ and $h_2(\zeta_1)=\zeta_2$. Thus, $h_1(\gamma_1)=\gamma_2$ and $h(\sigma_1)=\sigma_2$. We obtain a 1-parameter family   $\sigma_t=h_t(\sigma_1)$, $t\in [1,2]$ of chambers in $\St(\tau)$, connecting $\sigma_1, \sigma_2$.  
Thus, $\sigma_1, \sigma_2$ are in the same component of $\St(\eta)$. 

\medskip
(ii) This follows immediately from part (i).
\end{proof}

\begin{lemma}\label{lem:nontrivial projection}
Unless the projection of $\St(\eta)$ to $\Ft$ is a singleton, it is not contained in the antipodal set of any simplex $\tau_-\in {\rm Flag}(\iota\tmod)$.  
\end{lemma} 

Note that the condition that this projection is {\em not} a singleton amounts to saying that $\tmod$ is not contained in $\emod$.

\begin{proof}[Proof of \Cref*{lem:nontrivial projection}]
  The projection of $\St(\eta)$ to $\Ft$ consists of all simplices $\tau$ of type $\tmod$ contained in the star of $\eta$.  Consider one of these simplices, $\tau_+$, antipodal to a simplex $\tau_-$ (of type $\iota \tmod$). Then there exists an apartment $a$ in the spherical building of the symmetric space $X$, containing $\tau_-$ and the chamber $\sigma$ which, in turn contains  $\tau_+$ and $\eta$. Since $\tau_+$ is not contained in $\eta$, the stabilizer of $\eta$ in the Weyl group $W$ of $a$ does not fix $\tau_+$, there exists $w\in W$ fixing $\eta$ but not $\tau_+$. Thus, $w(\tau_+)=\tau$ also has type $\tmod$ and is contained  in the star of $\eta$ in $a$. However, an apartment cannot contain two simplices of the same type both antipodal to the same simplex. Thus, the simplex $\tau$ cannot be antipodal to $\tau_-$. 
\end{proof}

\begin{corollary}\label{thelemma} 
Suppose that $\eta$ is a simplex in the Tits building $\tits X$ of type $\emod$ which is not a face of $\tmod$. 
Then for any $\tau\in \Ft$,  $C(\tau)$ cannot contain the projection $\pi_{\tmod}(\St(\eta))$ for any simplex $\eta$ of type $\emod$. 

In particular, if $A$ and $B$ are nonempty sets in $\Ft$ which are antipodal to each other, then neither $A$ nor $B$ can 
 contain any stars as described above. 
 \end{corollary}

\subsection{Contraction and regularity}\label{sec:conreg}

Let $\tmod\subset\smod$ be a face (which is not necessarily $\iota$-invariant).
We equip
\[
 \bar X^\tmod \coloneqq X \sqcup \Ft.
\]
with the topology of {\em flag-convergence}; see \cite[Definition 3.90]{KLP:Morse}.
A sequence $(g_n)$ in $G$ is said to be {flag-converging} to $\tau_+\in \Ft$, denoted by $g_n\raf \tau_+$, if the orbit-sequence $(g_n x)$ flag-converges to $\tau_+$ in $\bar X^\tmod$.
This notion is independent of the choice of $x\in X$. In particular, if a sequence $(g_n)$ converges to $\tau_+$ then for $h\in G$, the sequence of compositions $(g_nh)$ also  flag-converges to $\tau_+$. At the same time, the sequence of compositions $(h g_n)$ flag-converges to $h(\tau_+)$. 

\begin{remark}\label{foot:regular}
We do not define the notion of a {\em $\tmod$-regular sequence} in this paper. We refer our reader to \cite[Definition 4.7]{MR3811766}. We note only that a sequence is {\em not} $\tmod$-regular  if and only if it contains a subsequence, which has no accumulation points in $\Ft$. 
Note that if a sequence $(g_n)$ in $G$ is $\tmod$-regular, then for any $h\in G$ the sequence $(hg_n)$ is also $\tmod$-regular. 
 Moreover, a sequence $(g_n)$ in $G$ is $\tmod$-regular if and only if $(g_n^{-1})$ is $\iota\tmod$-regular.
 In particular, if $\tmod$ is $\iota$-invariant, then $(g_n)$ is $\tmod$-regular  if and only if the inverse sequence $(g_n^{-1})$ is $\tmod$-regular; see the fourth paragraph of \cite[page 2560]{MR3811766}. 
 \end{remark}

\begin{proposition}\label{prop:cont_reg}
 Let $(g_n)$ be a sequence in $G$, and let $\tau_\pm\in {\rm Flag}(\pm\tmod)$. The following are equivalent:
\begin{enumerate}[label=(\roman*)]\itemsep0em
 \item As $n\ra\infty$,
 $
  g_n \vert_{C(\tau_-)} \ra \tau_+
 $
 uniformly on compacts.
 
 \item As $n\ra\infty$,
 $
  g_n^{-1} \vert_{C(\tau_+)} \ra \tau_-
 $
 uniformly on compacts.
 
 \item The sequence $(g_n)$ (equivalently, $(g_n^{-1})$) is $\tmod$-regular ($\iota\tmod$-regular), and $g_n^{\pm1} \raf \tau_\pm$ as $n\ra\infty$.
\end{enumerate}
\end{proposition}
 
\begin{proof}
For (i) $\iff$ (ii), see \cite[Lemma 4.4]{MR3811766}.

 For the direction (iii) $\implies$ (i), see \cite[Lemma 4.18]{MR3811766}.
 
 We now show that (i) and (ii) together imply (iii).
 For (i) implies $(g_n)$ is $\tmod$-regular (similarly, (ii) implies $(g^{-1}_n)$ is $\tmod$-regular), see \cite[Proposition 4.16]{MR3811766}.
 It remains to show that $g_n^\pm \raf \tau_\pm$.
 If $\hat\tau_+ \in \Ft$ is any flag-accumulation point for the sequence $(g_n)$, then there exists a subsequence $(g_{n_k})$ of $(g_n)$ such that $g_{n_k} \raf \hat\tau_+$.
 Applying the direction (iii) $\implies$ (i) and after extraction of $(g_{n_k})$, there exists $\hat\tau_-\in {\rm Flag}(\iota\tmod)$ such that $g_{n_k}\vert_{C(\hat\tau_-)} \ra \hat\tau_+$.
 By the uniqueness of the attractor, see \cite[Lemma 4.6]{MR3811766}, we get that $\tau_+ = \hat\tau_+$.
 Hence, $g_n \raf \tau_+$.
 Similarly, applying the direction (iii) $\implies$ (ii), it follows that $g_n^{-1} \raf \tau_-$.
\end{proof}

\subsection{Pure sequences}\label{sec:pure}

In this section, our discussion concerns {\em $\emod$-pure} sequences in $G$, where $\emod\subset\smod$ is a face.
We do not give a definition of $\emod$-pure sequences here as it requires some lengthier discussion, but we refer our readers to \cite[Definition 4.7]{MR3811766}.

\begin{remark}\label{rem:pure}
 The condition of being $\emod$-pure is stronger than being $\emod$-regular. Furthermore, if a divergent\footnote{That is, the sequence $(g_n)$ has no accumulation points in $G$.}
 sequence $(g_n)$ in $G$ {\em is not} $\tmod$-regular, then it contains a subsequence which is $\emod$-pure for some face $\emod$ of $\smod$ which does not contain 
$\tmod$. 
See \cite{MR3811766} for more details.
\end{remark}

The accumulation dynamics on the Furstenburg boundary of $\emod$-pure sequences is well-understood by \cite[Prop. 9.5]{MR3811766}; we record their result in the proposition below.
For $\eta\in {\rm Flag}(\emod)$, let $C_{\rm Fu}(\eta) \coloneqq \pi_{\emod}^{-1}(C(\eta))$. See \eqref{eqn:proj} for the definition of the map $\pi_\emod$.

\begin{proposition}\label{prop:KL9.5}
 Suppose that a sequence $(g_n)$ in $G$ is $\emod$-pure and $\emod$-contracting\footnote{That is, $g_n\vert_{C(\eta_-)} \to \eta_+$ uniformly on compacts.}
  for $\eta_\pm \in {\rm Flag}(\pm \emod)$.
 Then, after extraction,
 \[
  g_n\vert_{C_\Fu(\eta_-)} \ra \phi \quad \text{uniformly on compacts},
 \]
 where the map 
 \[
 \phi : C_\Fu(\eta_-) \ra \St(\eta_+)
 \]
 is an open (in manifold topology) algebraic map.
 Moreover, for every $\hat{\eta}\in C(\eta_-)$, the restriction $\phi\vert_{\St(\hat{\eta})}$ is given by the restriction of an element of $g\in G$, and hence is an algebraic isomorphism.
\end{proposition}

Suppose that we have a  $\emod$-pure sequence $(g_n)$ in $G$, such that
\[
 g_n^{\pm1} \raf \eta_\pm \in {\rm Flag}(\pm\emod),
\]
and which satisfies the 
the conclusion of the above proposition, i.e., $g_n\vert_{C_\Fu(\eta_-)} \ra \phi$ uniformly on compacts as $n\ra\infty$.
Consider a point $\eta\in {\rm Flag}(\emod)$ antipodal to $\eta_-$.
Then, $\St(\eta) \subset C_\Fu(\eta_-)$.
The proposition above implies that, for large $n$, $g_n : \St(\eta) \ra \Fs$ {\em approximates} the algebraic isomorphism $\phi: \St(\eta) \ra \St(\eta_+)$.

\begin{figure}[h]
\centering
\begin{overpic}[scale=.65,tics=5]{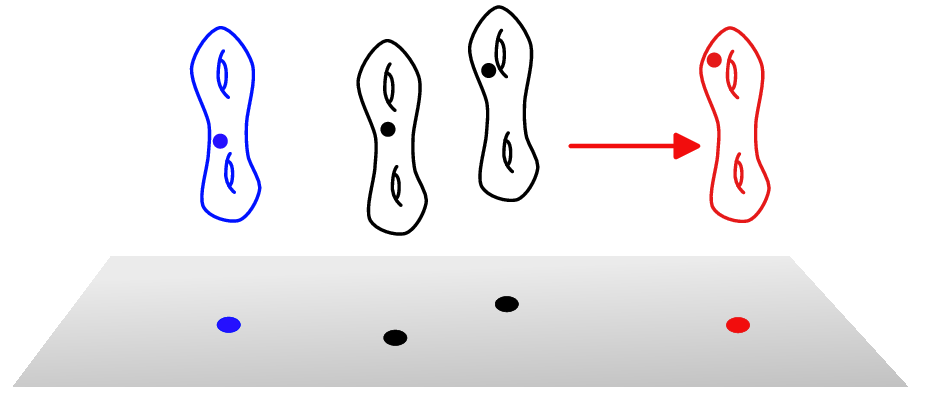}
\put(-2,16.5){\rotatebox{90}{$\xleftarrow[]{\pi_\emod}$}}
\put(-8,28){\small$\Fs$}
\put(-8,13){\small${\rm Flag}(\emod)$}
\put(19,25){\rotatebox{90}{\tiny\color{blue}$\St(\eta)$}}
\put(83,25){\rotatebox{90}{\tiny\color{red}$\St(\eta_+)$}}
\put(67,30){\large\color{red}$\phi$}
\put(42,26.5){\small$x$}
\put(72,38){\color{red}\small$x_+$}
\put(24,5){\color{blue}$\eta$}
\put(40,3){$g_n\eta$}
\put(51.5,7){$g_m\eta$}
\put(78.5,5){\color{red}$\eta_+$}
\end{overpic}
\caption{}\label{fig:phi}
\end{figure}

After fixing a distance function $d$ on $\Fs$ which is compatible with the manifold topology, we consider the quantity
\[
 D_{m,n} \coloneqq \max_{y\in \St(\eta)} d(g_m y,g_n y) 
 =  \max_{x\in \St(g_n \eta)} d(g_m g_n^{-1} x, x).
\]
See \Cref{fig:phi} for an illustration.

\begin{lemma}\label{lem:stabilzes}
$D_{m,n} \ra 0$ as $m, n\ra \infty$.
\end{lemma}

\begin{proof}
 Suppose that, on the contrary, there exist sequences of natural numbers $(m_k)$, and $(n_k)$ such that $m_k, n_k \ra \infty$ as $k\ra\infty$, for which the following holds: For all $k\in \N$, there exists $x_{k} \in \St(g_{n_k} \eta)$ such that
 \begin{equation}\label{eqn:dist_pos}
  d(g_{m_{k}} g_{n_k}^{-1} (x_{k}), x_{k}) \ge \epsilon >0.
 \end{equation}
 
 Note that the sequence $(x_k)$ accumulates in $\St(\eta_+)$, and $y_k \coloneqq g_{n_k}^{-1} (x_{k}) \in \St(\eta)$.
 Passing to a subsequence of  the sequence $(n_k)$, we may assume that the sequence $x_k \ra x_+ \in \St(\eta_+)$, and $y_k \ra y_0 \in \St(\eta)$, as $k\ra \infty$.
 
We will show that $\phi(y_0) = \lim_{m\ra \infty} g_{m} y_0 = x_+$.
Consider the compact set $A = \{ y_i \mid i\in\N\} \cup\{y_0\}\subset\St(\eta)$.
Then,
\[
 g_m\vert_A \ra \phi\vert_A, \quad \text{as } m\ra\infty.
\]
Since the convergence is uniform, $g_m y_i$ converges to $\phi(y_0)$ whenever $i\ra \infty$ and $m \ra \infty$.
Now, since $x_k = g_{n_k}y_k$, and $\lim_{k\ra\infty} n_k\ra\infty$, we obtain that
\begin{align*}
\phi(y_0) &
 =\lim_{i,m\ra \infty} g_{m} y_i\\
 &
 =\lim_{k\ra\infty} g_{n_k}y_k = \lim_{k\ra\infty} x_k = x_+.
\end{align*}

In particular, we get $g_{m_k} g_{n_k}^{-1}  (x_{k}) =  g_{m_k} y_k \ra x_+$ as $k\ra\infty$.
Since also $x_k \ra x_+$, we must have
\[
 \lim_{k\ra\infty}d(g_{m_k} g_{n_k}^{-1}  (x_{k}), x_k) = 0.
\]
This contradicts \eqref{eqn:dist_pos}.
\end{proof}

\subsection{A regularity criterion}\label{sec:regulairty}

Let $\tmod\subset\smod$ be a face, which is not assumed to be $\iota$-invariant.
We equip $\Ft$ with a distance function $d$ which is compatible with its manifold topology.
Note that since $\Ft$ is compact, the notion of the convergence $A_n\to \tau$ defined in \Cref{sec:convergence}, where $(A_n)$ is a sequence of nonempty compact subsets of $\Ft$ and $\tau\in\Ft$,  does not depend on the choice of such a distance function $d$.

\begin{lemma}\label{prop:regulairty}
 Let $(g_n)$ be a sequence in $G$.
 If there exists a compact subset $A \subset \Ft$ with nonempty interior and a point $\tau_+\in \Ft$ such that $g_n A \ra \tau_+$, then $(g_n)$ is $\tmod$-regular and  $(g_n)$ flag-converges to $\tau_+$.
\end{lemma}

\begin{proof}
Since $g_n A \ra \tau_+$ as $n\to\infty$ and $A$ has nonempty interior, it is clear that the sequence $(g_n)$ is divergent in $G$.
 Therefore, if $(g_n)$ is not $\tmod$-regular, then, after extraction, $(g_n)$ is $\emod$-pure for some $\emod \subset \smod$ which does not contain $\tmod$; see \Cref{rem:pure}.
 After further extraction, there exists $\eta_\pm\in {\rm Flag}(\pm\emod)$ such that $(g_n)$ is $\emod$ contracting for $\eta_\pm$.
 Let $\phi : C_\Fu(\eta_-) \ra \St(\eta_+) $ be the surjective algebraic map that we obtain from \Cref{prop:KL9.5}.
 Since $\widetilde A \coloneqq \pi_\tmod^{-1}(A)$ is Zariski dense (since it has nonempty interior), $\phi(\widetilde A)$ is also Zariski dense in $\St(\eta_+)$.
 Now, under the hypothesis $g_n A \ra \tau_+$, we obtain that
 \[
  g_n \widetilde A \text{ accumulates on } \St(\tau_+), \quad \text{as } n\ra\infty.
 \]
 Hence, we must have $\St(\tau_+) \supset \St(\eta_+)$, which leads to $\tmod \subset \emod$.
 This is a contradiction.
 
 The convergence  $g_n \raf \tau_+$ follows from \Cref{prop:cont_reg}.
\end{proof}

\begin{corollary}\label{prop:regulairty_metric}
 Let $(g_n)$ be a sequence in $G$.
 If there exists a compact subset $A \subset \Ft$ with nonempty interior such that 
 \[
  \lim_{n\ra\infty}{\rm diam}\, g_n A = 0,
 \]
 then $(g_n)$ is a  $\tmod$-regular sequence.
\end{corollary}

\begin{proof}
 If ${\rm diam}\, g_n A \ra 0$ as $n\ra\infty$, then, after extraction, $g_n A \ra \tau_+$ for some $\tau_+ \in \Ft$.
 Then, \Cref{prop:regulairty} implies that $(g_n)$ is $\tmod$-regular.
\end{proof}

\subsection{Regular and Anosov subgroups}\label{sec:discrete_groups}

In this paper, we consider the following classes of discrete subgroups of $G$.

Let $\tmod \subset \smod$ be an $\iota$-invariant face.
A discrete subgroup $\G < G$ is called {\em $\tmod$-regular} if all sequences of distinct elements $(\g_n)$ in $\G$ are $\tmod$-regular.
If $\G$ is such a subgroup, then the {\em $\tmod$-limit set}
\[
 \Lambda_\tmod(\G) \subset \Ft
\]
consists of all points $\tau\in\Ft$ such that there exists a sequence $(\g_n)$ in $\G$ for which $\g_n\raf \tau$.
This is a $\G$-invariant closed subset of $\Ft$.

The class of {\em $\tmod$-Anosov subgroups} of $G$, a notion originally introduced by Labourie \cite{MR2221137}, is a special class of regular subgroups that generalizes the {\em convex cocompact} Kleinian groups into higher rank.

\begin{remark}
 There are several equivalent characterizations of $\tmod$-Anosov subgroups. For our purpose, we use the following characterization of $\tmod$-Anosov subgroups as {\em $\tmod$-asymptotically embedded} subgroups, which was introduced by Kapovich, Leeb, and Porti in \cite{MR3736790}.
 See \cite[Theorem 5.47]{MR3736790} for the equivalence between these two (and several other) definitions.
\end{remark}

\begin{definition}[Asymptotically embedded]\label{def:AE}
 A subgroup $\G < G$ is called {\em $\tmod$-asymptotically embedded} if it is $\tmod$-regular and satisfies the following conditions:
\begin{enumerate}[label=(\roman*)]\itemsep0em
 \item $\G$ as an abstract group is word-hyperbolic.
 \item The $\tmod$-flag limit set $\Lt(\G)$ is antipodal.
 \item There exists a $\G$-equivariant homeomorphism
 \[
  \xi : \vb\G \ra \Lt(\G)
 \]
 from the Gromov boundary of $\G$ onto its $\tmod$-flag limit set, which continuously extends the orbit map $o_x: \Gamma\to \Gamma x\subset X$. 
\end{enumerate}
\end{definition}

Note that the antipodality of $\Lt(\G)$ and injectivity of $\xi$ is equivalent to the condition that the map $\xi$ is {\em antipodal}, i.e. it sends distinct elements of $\vb\G$ to antipodal elements of $\Ft$. 

\section{Discreteness}\label{sec:discreteness}

In what follows, $\tmod$ will always denote an $\iota$-invariant face of $\smod$.

\begin{definition}\label{defn:ping-pong}
Let $\G_A$ and $\G_B$ be a pair of discrete subgroups of $G$.
A pair $(A,B)$ of disjoint compact subsets of $\Ft$ is called a {\em ping-pong pair} for $(\G_A,\G_B)$ if the following conditions are satisfied:

\begin{enumerate}[label=(\roman*)]\itemsep0em
 \item $A$ and $B$ have nonempty interiors.
 \item For all nontrivial elements $\alpha\in \G_A$ and $\beta\in \G_B$, we have that $\alpha B \subset \Int A$ and $\beta A \subset \Int B$.
\end{enumerate}
\end{definition}

Let $\G_A$ and $\G_B$ are discrete subgroups of $G$ such that $(\G_A,\G_B)$ admits a ping-pong pair $(A,B)$ in $\Ft$.
The classical {\em ping-pong argument} of Felix Klein (cf. \cite{tits1972free}), shows that the subgroup $\G \coloneqq \inp*{\G_A}{\G_B}$ generated by $\G_A$ and $\G_B$ in $G$ is {\em naturally} isomorphic to the  free product $\G_A*\G_B$:  

\begin{lemma}[Ping-pong]\label{lem:ping-pong}
 Let $\G_A$ and $\G_B$ be two discrete subgroups of $G$ such that $(\G_A, \G_B)$ admits a ping-pong pair $(A,B)$ in $\Ft$.
 Then, the natural homomorphism $\phi: \G_A * \G_B \ra G$ extending the embeddings $\G_A\to G, \G_B\to G$, 
 is injective.
\end{lemma}

\begin{proof}
In the  free product $\G_A * \G_B$, pick any nontrivial element $\g$.

Suppose that the rightmost letter of the reduced form of $\g$ is in $\G_A$ and the leftmost letter of the reduced form of $\g$ is in $\G_B$ (see \Cref{sec:reduced}).
Then, $\phi(\g) A \subset \Int A$. Hence, $\phi(\g) \ne 1 \in G$.
Similarly if the rightmost/leftmost letters of the reduced form of $\g$ are both in $\G_A$, then $\phi(\g) B  \subset \Int A$. Hence, $\phi(\g) \ne 1 \in G$ as well. The verification of the non-triviality of $\phi(\g)$ in the other two possibilities on the location of rightmost/leftmost letters in the reduced form of $\g$ is similar; we leave the details to the reader.
\end{proof}

In what follows, we reserve the notation $\G$ to denote the subgroup of $G$ generated by $\G_A$ and $\G_B$, 
\[
 \G \coloneqq \inp*{\G_A}{\G_B} \cong \G_A * \G_B.
\]

Next, we impose the $\tmod$-regularity condition for the groups $\G_A$ and $\G_B$.
Recall the notion of the $\tmod$-flag limit set from  \Cref{sec:discrete_groups}.
The lemma below shows that the limit sets of $\G_A$ and $\G_B$ are contained in $A$ and $B$, respectively.

\begin{lemma}\label{lem:Lt}
 Let $\G_A,\G_B, A$, and $B$ be as in \Cref{lem:ping-pong}.
 Assume that $\G_A$ (resp. $\G_B$) is $\tmod$-regular.
 Then,
 \[
  \Lt(\G_A) \subset A \quad (\text{resp. }
  \Lt(\G_B) \subset B).
 \]
\end{lemma}

\begin{proof}
Let us verify that $\Lt(\G_A) \subset A$. The verification of $\Lt(\G_B) \subset B$ is similar.

Pick arbitrary point $\tau_+ \in \Lt(\G_A)$.
There exists a sequence $(\alpha_n)$ in $\G_A$ such that $\alpha_n^{\pm1} \raf \tau_\pm$, where $\tau_-$ is some point in the limit set $\Lt(\G_A)$.
Since $B$ 
has nonempty interior, and $C(\tau_-)$ is an open dense subset of $\Ft$, 
$B\cap C(\tau_-)  \ne \emptyset$. 
Pick any point $\hat\tau\in B \cap C(\tau_-)$.
Then, by \Cref{prop:cont_reg},  $\lim_{n\ra\infty} \alpha_n \hat\tau = \tau_+$.
On the other hand, for all $n\in\N$, $\alpha_n \hat\tau\in \alpha_n(B)\subset  A$, as long as $\alpha_n\ne 1$.
By the compactness of $A$, we get $\tau_+\in A$. Hence, $\Lt(\G_A) \subset A$.
\end{proof}

Now, we state and prove the main result of this section.

\begin{proposition}[Discreteness]\label{prop:discrete}
Let $\G_A$ and $\G_B$ be $\tmod$-regular subgroups of $G$ such that $(\G_A, \G_B)$ admits a ping-pong pair $(A,B)$ in $\Ft$.
If $A$ is antipodal to $B$, 
then the group $\G\coloneqq\inp*{\G_A}{\G_B} \cong \G_A * \G_B$ is a discrete subgroup of $G$.
\end{proposition}

\begin{proof}
 Suppose, to the contrary, that  $\G$ is not discrete.
 Then, there exists a sequence $(\g_n)_{n\in\N}$ of distinct elements of $\G$ such that $\g_n\ra 1 \in G$. 
  After passing to a subsequence, we may assume that
\begin{enumerate}[label=(\roman*)]\itemsep0em
 \item either, for all $n\in \N$, the rightmost letter in the reduced form of $\g_n$ (see \Cref{sec:reduced}) is in $\G_A$,
 \item or, for all $n\in \N$, the rightmost letter in the reduced form of $\g_n$ is in $\G_B$.
\end{enumerate}
Since the two cases differ by relabelling, we make the additional assumption that (i) holds. By the assumption $\g_n \to 1$, we have
 \begin{equation}\label{eqn:gkpk2}
  \g_n B \ra B, 
 \end{equation}
 i.e., the Hausdorff distance between $\g_n B$ and $B$ converges to zero as $n\to \infty$.

 Since for all $b\in B$, $\lim_{n\ra\infty} \g_n b = b$, we must have the following: For all large $n\in\N$, the (nontrivial) leftmost letter  of $\g_n$ is  
 $\beta_n\in \G_B$. 
 In particular, for all large $n\in\N$,
 \begin{equation}\label{eqn:gkpk}
  \beta_{n}^{-1} \g_n B \subset  A.
 \end{equation}
 
\setcounter{case}{0}
\begin{case}
 Suppose that the sequence $( \beta_{n})$ is bounded in $G$.
 After extraction, we may assume that $\beta_{n} \equiv \beta\in\G_B \setminus \{1\}$, i.e., $( \beta_{n})$ is a constant sequence.
By \eqref{eqn:gkpk}, for all large $n\in\N$,
\[
 \g_n B \subset  \beta A \subset \Int B \subset B.
\]
In particular, any accumulation point in $\Ft$ of the sequence of subsets $(\g_n B)$ is contained in $\beta A$, which is {properly} contained in $B$.
This contradicts \eqref{eqn:gkpk2}.
\end{case}
 
\begin{case}
 Suppose now that the leftmost letter sequence $( \beta_{n})$ associated to $(\g_n)$ is unbounded in $G$.
Since the sequence $( \beta_{n})$ comes from $\G_B$, and $\G_B$ is $\tmod$-regular,
after  extraction of the sequence, we obtain that there exist points $b_\pm\in\Lt(\G_B)$ such that
\[
 \beta_{n}\vert_{C(b_-)} \ra {b_+}, \quad \text{as } n\ra\infty,
\]
uniformly on compact subsets.
Recall that, by  \Cref{lem:Lt}, $\Lt(\G_B) \subset B$.

Since $A$ and $B$ are antipodal to each other and $b_-\in B$ (see \Cref{lem:Lt}), we obtain that $A \subset C(b_-)$; cf. \eqref{def:cell}.
Then, by \Cref{prop:cont_reg}, $\beta_{n} A \ra b_+$ as $n\ra\infty$. 
Moreover, by \eqref{eqn:gkpk},
\[
 \g_n B \subset \beta_{n} A, \quad \forall n\gg 1.
\]
Hence, we also have that $\g_n B \ra b_+$ as $n\ra\infty$. 
This contradicts \eqref{eqn:gkpk2}.
\end{case}

Combining the above two cases, we complete the proof of this proposition.
\end{proof}

\section{Regularity}\label{sec:regularity}

The main result of this section stated below shows that the subgroup $\G$ in the conclusion of \Cref{thm:main} is $\tmod$-regular.

\begin{theorem}[Regularity]\label{thm:regularity}
 Let $\G_A$ and $\G_B$ be $\tmod$-regular subgroups of $G$.
 Suppose that $(\G_A,\G_B)$ admits a ping-pong pair $(A,B)$ in $\Ft$.
 If  $A$ is antipodal to  $B$, then $\G \coloneqq \inp*{\G_A}{\G_B} < G$ is a $\tmod$-regular subgroup.
\end{theorem}

In special the case $G = \isom(\H^n)$, the regularity of subgroups of $G$ is equivalent to discreteness.
Recall that, for general semisimple Lie groups $G$, we have already established the discreteness of $\G$ in the conclusion of the above result; see \Cref{prop:discrete}.
However, if $G$ is of higher rank (e.g., $G = \PSL(3,\R)$), then regularity of subgroups is a much stronger property than just discreteness.
Therefore, \Cref{thm:regularity} is nontrivial when $G$ is of higher rank.

Before discussing the proof of \Cref{thm:regularity}, we introduce the following definition:

\begin{definition}
A sequence of distinct elements of $G$ is  {\em $\tmod$-irregular}  if it contains no $\tmod$-regular subsequences. 
\end{definition}

Note that a discrete subgroup $\G< G$ is $\tmod$-regular if and only if it contains no $\tmod$-irregular sequences. 

\Cref{thm:regularity} is a key step in the proof of our main theorem (\Cref{thm:main}) and the proof of it occupies the rest of this section.
Here we lay out our plan for the proof of \Cref{thm:regularity}. The main technical step in the proof is to show that a certain kind of sequences in $\G$ (the resulting subgroup in the conclusion of \Cref{thm:regularity}), which are called {\em special}, are $\tmod$-regular; see \Cref{claim:two}.
The proof of this lemma relies upon several results established in \Cref{sec:flag_var,sec:pure,sec:regulairty}.
An immediate consequence of this lemma is that {\em alternating} sequences in $\G$ are $\tmod$-regular, which is then used (with the help of \Cref{prop:cont_reg}) to show that  certain sequences of  nested images of the subsets $A, B$ under alternating sequences have singleton intersections; see  \Cref{lem:int_singleton} for a precise statement.
With this knowledge, we finish the proof of \Cref{thm:regularity} by contradiction:
Assuming $\G$ is not $\tmod$-regular, we may obtain a $\tmod$-irregular sequence $(\g_n')$ in $\G$.
Using the $\tmod$-irregularity property, we are able to find a subsequence $(\g_n)$ of $(\g_n')$ and a special sequence $(\omega_n)$ in $\G$ such that for all $n\in\N$, $\omega_n$ is a rightmost subword of $\g_n$.
Then, with the help of \Cref{lem:int_singleton}, we can easily extract a subsequence of $(\g_n^{-1})$ under which the images of some (fixed) compact subset of $\Ft$ with nonempty interior converge to a point, which would imply (by \Cref{prop:regulairty}) that this subsequence (hence, its inverse sequence, which is a subsequence of the original sequence $(\g_n')$) is $\tmod$-regular; see \Cref{lem:final}. This would be a contradiction to the assumption of $\tmod$-irregularity of $(\g_n')$.

We now discuss our proof in detail:
 
 \begin{proof}[Proof of \Cref*{thm:regularity}]
 By \Cref{prop:discrete}, we know that $\G$ is discrete.
 Thus, we have to show that $\G$ does not contain any $\tmod$-irregular sequences consisting of distinct elements.
 We first show that such 
 sequences in $\G$ cannot have uniformly bounded relative length (see \Cref{sec:reduced} for the definition).

\begin{lemma}\label{lem:bdd_length}
 If a  sequence of distinct elements of $\G$ has uniformly bounded relative length, then it  is $\tmod$-regular.
\end{lemma}

\begin{proof}
Let $(\g_n)$ be such a sequence.
After passing to a subsequence, we can (and will) assume that there exists $N\in\N$ such that, for all $n\in \N$, $\rell(\g_n) = N$.
If $N = 1$, then clearly the sequence lies in $\G_A \cup \G_B$, and hence it is $\tmod$-regular (since both subgroups $\G_A, \G_B$ of $G$ are assumed to be $\tmod$-regular). 

We prove the claim by an induction on $N$.
Suppose that any sequence in $\G$ whose relative length is constant and $<N$ is $\tmod$-regular.

Let $(\g_n)$ be a sequence such that, for all $n\in\N$, $\rell(\g_n) = N$.
Deleting the leftmost letter $\delta_n \in \G_A \cup \G_B$ from the reduced form of each element in $(\g_n)$, we obtain a new sequence $(\g_n')$, i.e., for each $n\in\N$, $\g_n = \delta_n \g_n'$.
We consider two possibilities.

\setcounter{case}{0}
\begin{case}
 Suppose that $\{\delta_n \mid n\in \N\}$ is finite. In this case, using the induction hypothesis $(\g_n')$ and, consequently $(\g_n)$, is  
$\tmod$-regular (see the discussion in \Cref{sec:conreg}). 
\end{case}

\begin{case}\label{lem3.3case2}
 Thus, we now assume that $\{\delta_n \mid n\in \N\} \subset \G_A \cup \G_B$ infinite. 
Then, after passing to a subsequence of $(\g_n)$, $(\delta_n)$ is $\tmod$-regular and, either $(\delta_n)$ is a sequence in $\G_A$, or a sequence in $\G_B$. It suffices to consider the former case, i.e., $(\delta_n)$ is a sequence in $\G_A$ (the other case is obtained by relabelling). Moreover, after further extraction of $(\g_n)$, we also assume that the rightmost letters of the reduced form of all the elements of $(\g_n')$ come from either $\G_A$ or $\G_B$; suppose that all the rightmost letters are in $\G_A$ (again, the other case when all the rightmost letters are in $\G_B$ is similar, and we skip that case).
Since $(\delta_n)$ is $\tmod$-regular, after passing to a subsequence, $\delta_n^{\pm1} \raf d_\pm \in A$.
Since $B\subset C(d_-)$, we obtain
\[
 \g_n (B) \subset \delta_n(B) \ra d_+, \quad \text{as } n\ra\infty.
\]
Using  \Cref{prop:regulairty}, the above shows that $(\g_n)$ is $\tmod$-regular.\qedhere
\end{case}
\end{proof}

By the above lemma, our study of $\tmod$-irregular sequences $(\g_n)$ in $\G$ is reduced to the case when $(\g_n)$ has unbounded relative length.
In particular, if $(\g_n)$ is $\tmod$-irregular, then
after passing to a subsequence,
we further assume that
\begin{equation}\label{eqn:length_increasing}
 \rell(\g_n) < \rell(\g_{n+1}),\quad
 \forall n\in\N.
\end{equation}

We next study the $\tmod$-regularity property of a certain kind of sequences which we call {\em special}: A sequence $(\g_n)$ in $\G$ is called {\em special} if it satisfies \eqref{eqn:length_increasing}, and, for all $n\in\N$,  $\g_n$ is a rightmost subword of $\g_{n+1}$ (see \Cref{sec:reduced} for the definition).

\begin{lemma}\label{claim:two}
 Special sequences in $\G$ are $\tmod$-regular.
\end{lemma}

\begin{proof}
Let $(\g_n)$ be a special sequence.
 Suppose that the rightmost  and  leftmost letters of the reduced form of each element $\g_n$ in the sequence are in $\G_A$ (there are three more possibilities, and the analysis in each case is similar). 
 After extraction, we reduce to one of the following two cases:
 
 \setcounter{case}{0}
 \begin{case}\label{case:one}
Assume that the leftmost letter sequence $(\alpha_n)$ of $(\g_n)$ is unbounded.
In this case, after passing to a subsequence of $(\g_n)$, we assume that the leftmost letter sequence $(\alpha_n)$ of  $(\g_n)$ flag-converges:
\[
 \exists a_\pm \in A \text{ such that }
 \alpha_n^{\pm1}\raf a_\pm, \quad \text{as } n\ra\infty. 
\]
Since the $B$ is antipodal to $a_-$,  we deduce that
\[
 \gamma_n(B) \subset \alpha_n(B) \ra a_+, \quad \text{as } n\ra\infty. 
\]
Applying \Cref{prop:regulairty}, we get that $(\g_n)$ is $\tmod$-regular.
\end{case}
 
\begin{case}
 Next, we look at the complementary case: For all $n\in \N$, the leftmost letters of the reduced form of all $\g_n$ come from a finite subset $S\subset \G_A$.

 Note that, since $(\g_n)$ is special, for each $n\in\N$, the rightmost letter  of $\g_{n+1} \g_n^{-1}$ is in $\G_B$,  and
\begin{equation}\label{eqn:nesting_case1}
\g_{n+1} \g_n^{-1} \widetilde A \subset S \widetilde B
 \subset \Int\widetilde A,
\end{equation}
where
\[
 \widetilde A = \bigcup_{a\in A} \St(a), \quad 
 \widetilde B =\bigcup_{b\in B} \St(b).
\]
Note that since $S\widetilde B$ is a union of a finite number of compact sets, the subset $S\widetilde B$ is compact.
We observe that $(\widetilde A, \widetilde B)$ is a ping-pong pair for  $(\G_A,\G_B)$ in $\Fs$.

Suppose that $(\g_n)$ is not $\tmod$-regular.
After extraction, the sequence $(\g_n)$ is $\emod$-pure, and
$
 \g_i^{\pm1} \raf \eta_\pm,
$
for some $\eta_\pm \in {\rm Flag}(\pm\emod)$,
where $\emod \subset \smod$ is a face such that 
\begin{equation}\label{eqn:containment}
\tmod \not\subset \emod,
\end{equation}
see \Cref{rem:pure}.
Recall from \Cref{sec:pure} that, after further extraction of $(\g_n)$, there exists a surjective algebraic map $\phi : C_\Fu(\eta_-) \ra \St(\eta_+)$ such that
$
  \g_k\vert_{C_\Fu(\eta_-)} \ra \phi 
$
uniformly on compacts.
See \Cref{prop:KL9.5}.

We pick a point $z\in C_\Fu (\eta_-) \cap \widetilde A$.\footnote{Since $\widetilde A$ has nonempty interior,  $C_\Fu (\eta_-) \cap \widetilde A\ne \emptyset$.}
Then, $z$ is contained in a subvariety
$\St(\eta)$, for some (unique) $\eta\in C(\eta_-)$.
Since, for all $n\in\N$, $\g_{n}(z) \in  S \widetilde B$, and $S \widetilde B \subset \Int \widetilde A$,
it follows that
$\St(\g_{n} \eta)$ 
intersects the interior of $\widetilde A$.

Let us fix a background distance function $d$ on $\Fs$ which is compatible with the manifold topology.
By \Cref{lem:stabilzes},
\begin{equation}\label{eqn:stabilzes}
 D_{n+1,n} 
 =  \max_{x\in \St(\g_{n}\eta)} d(\g_{n+1} \g_{n}^{-1} x, x)
 \ra 0,
 \quad \text{as } n\ra \infty.
\end{equation}

Next, we observe that, for all $n\in\N$, $\St(\g_{n}\eta) \not\subset \widetilde A$: For if $\St(\g_{n}\eta) \subset \widetilde A$, then we must have that $\pi_{\tmod} (\St(\g_{n}\eta)) \subset A$, where $\pi_\tmod$ is the projection map given by \eqref{eqn:proj}.
In this case, in view of \eqref{eqn:containment}, 
we obtain a contradiction with  \Cref{thelemma}. 

Since  $\St(\g_{n}\eta)$ is connected (see  \Cref{lem:connected fibers}), 
the preceding paragraph ensures that the intersections $\St(\g_{n}\eta) \cap \partial\widetilde A$ are nonempty, where $\partial \widetilde A$ denotes the frontier of the subset $\widetilde A$ of $\Fs$. 
For each $n\in \N$, choose $x_n  \in \St(\g_{n}\eta)  \cap \partial\widetilde A$.
By \eqref{eqn:stabilzes} and the triangle inequality, we have
\begin{equation}\label{eqn:accumulates}
 \lim_{n\ra\infty} d(\g_{n+1} \g_{n}^{-1} x_n , \partial\widetilde A) = 0.
\end{equation}
However, since $x_n \in \widetilde A$, by \eqref{eqn:nesting_case1},
the quantity
$d(\g_{n+1} \g_{n}^{-1} x_n , \partial\widetilde A)$
must be uniformly (over $n$) bounded below by a positive number.
This contradicts \eqref{eqn:accumulates}.
\end{case}

Combining the above two cases, we complete the proof of the lemma. 
\end{proof}

A sequence $(\omega_n)$ in $\G$ is called {\em alternating} if, there exist sequences $(\alpha_n)$ in $\G_A \setminus \{1\}$ and $(\beta_n)$ in $\G_B \setminus \{1\}$ such that
\begin{align}\label{eqn:specialform}
\begin{split}
\text{Type A} \quad &: \quad \text{Either}, ~\forall n\in\N, ~\omega_n = \alpha_1\beta_1\dots \alpha_{n-1}\beta_{n-1}\alpha_n,\\
\text{Type B} \quad &: \quad \text{Or}, ~\forall n\in\N,~ \omega_n = \beta_1\alpha_1\dots \beta_{n}\alpha_{n}.
\end{split}
\end{align}
Applying $\omega_n$ to $B$, we obtain a
sequence of compact subsets $(\omega_n B)$ of $\Ft$.
Note that, for all $n\in\N$, $\omega_{n+1} B = \omega_n \beta_n(\alpha_{n+1} B) \subset \omega_n (\beta_n A) \subset \omega_n (B)$; cf. item (ii) in \Cref{defn:ping-pong}.
Thus, we obtain a nested sequence of compact subsets
\begin{equation}\label{eqn:nested}
  \omega_1 B \supset \omega_2 B \supset \cdots.
\end{equation}
Moreover, by definition of being alternating, for all $n\in\N$, $\omega_n$ is a leftmost subword of $\omega_{n+1}$, and therefore, $(\omega_n^{-1})$ is special.

\begin{lemma}\label{lem:int_singleton}
If $(\omega_n)$ is alternating, then $(\omega_n)$ is $\tmod$-regular and intersection 
 \[\bigcap_{i\in\N} \omega_n B\] is singleton. In particular, as $n\ra\infty$, $\omega_n B$ converges to a point in $\Ft$.
\end{lemma}

\begin{proof}
We first assume that $(\omega_n)$ is of type A.

The first claim about $\tmod$-regularity follows directly by the observation that $(\omega_n^{-1})$ is special, and then by applying \Cref{claim:two} to conclude that $(\omega_n^{-1})$ is $\tmod$-regular.
Hence, $(\omega_n)$ is also $\tmod$-regular, see \Cref{foot:regular}.

 By $\tmod$-regularity, there exists a subsequence, $(\omega_{n_k})$ of $(\omega_n)$ and $\tau_\pm \in \Ft$ such that $\omega_{n_k}^{\pm1} \raf \tau_\pm$. By \Cref{prop:cont_reg}, as $k\ra\infty$,
 \[
  \omega^{\pm1}_{n_k} \vert_{C(\tau_\mp)} \ra \tau_\pm, \quad
  \text{uniformly on compacts.}
 \]
 Because of the special form of $\omega_n$ given in \eqref{eqn:specialform}, we observe  that $\tau_\pm\in A$: For if $\hat\tau\in B$ is any point antipodal to the pair $\tau_\pm$, then, for all $k\in\N$, $\omega_{n_k} \hat\tau\in \omega_{n_k} B \subset A$, where the last inclusion follows from the fact that $(\omega_n)$ is of type A.
 In particular, $\omega_{n_k}\hat\tau \in A$.
 Furthermore, since $\omega_{n_k} \hat\tau \ra \tau_+$ and $A$ is compact, we obtain that $\tau_+\in A$. 
 On the other hand, since we also have $\omega_{n_k}^{-1} B \subset A$, 
  applying a similar argument, we get that $\tau_-\in A$ as well.
 
 In particular, $B\subset C(\tau_-)$.
 Hence, $\omega_{n_k} B \to \tau_+$, as $k\ra\infty$.
 The lemma then follows by the nesting  in \eqref{eqn:nested}.
 
 \medskip
 Next, we consider the complementary case that $(\omega_n)$ is of type B.
 However, this case follows from the previous one by observing that $(\beta_1^{-1} \omega_n)$ is of type A, where $\beta_1\in\G_B$ is the common leftmost letter of the reduced form of the elements in $(\omega_n)$. 
\end{proof}

Now we finish the proof of  \Cref{thm:regularity}.
Suppose, to the contrary, that $(\g_n')$ is a $\tmod$-irregular sequence in $\G$.
We can (and will) assume that $(\g_n')$ satisfies   \eqref{eqn:length_increasing}, cf. \Cref{lem:bdd_length} above.

We construct two sequences $(\omega_n)$ and $(\g_n)$ in $\G$, such that $(\omega_n)$ is special, $(\g_n)$ is a subsequence of $(\g_n')$, and for each $n\in\N$, $\omega_n$ is a rightmost subword of $\g_n$:
After passing to a subsequence of $(\g_n')$, we assume that the rightmost letters $\alpha_n$ of the reduced form of the elements of $(\g_n')$ are all in either  $\G_A$ or $\G_B$. We assume that $\alpha_n \in \G_A$ for all $n\in\N$; the other possibility is analyzed by relabelling.
Note that the rightmost letter sequence $(\alpha_n)$ must be bounded in $G$ because, otherwise, by \Cref{case:one} in the proof of \Cref{claim:two}, $({\g'_n}^{-1})$, and hence $(\g'_n)$, would contain a $\tmod$-regular sequence.
Thus, $(\g'_n)$ has a subsequence $(\g'_{n_i})$ such that the rightmost letters of the reduced form of $\g'_{n_i}$'s are all the same, say all these are equal to $\alpha_1\in\G_A$.
Set $\omega_1 \coloneqq \alpha_1$ and $\g_1 \coloneqq \g_{n_1}'$.
Applying a similar argument (to the sequence $(\g'_{n_i}\alpha_1^{-1})_{i\in\N}$), we obtain a further subsequence $(\g'_{n_{i_k}})$ of $(\g'_{n_i})$ such that second letters from the right in the reduced form of $\g'_{n_{i_k}}$'s are all the same, say all these are equal to $\beta_1\in\G_B$.
Set $\omega_2 \coloneqq \beta_1\alpha_1$ and $\g_2 \coloneqq \g'_{n_{i_1}}$.
Proceeding inductively, we construct a special sequence $(\omega_n)$ in $\G$ and a subsequence of $(\g_n)$ of $(\g_n')$ such that
\[
 \omega_n \text{ is an rightmost subword of } \g_n, \quad \forall n\in\N.
\]
In particular, for all $n\in\N$, $\g_n = \delta_n \omega_n$, where $\delta_n = \g_n \omega_n^{-1}$ is a leftmost subword of $\g_n$.
After an extraction of $(\g_n)$, we can assume that the leftmost letters of the reduced form of the elements of $(\delta_n)$ all come from the same group, say $\G_A$. 
Moreover, by the construction, the sequence $(\omega^{-1}_{2n+1})_{n\in\N}$ is type A alternating (see \eqref{eqn:specialform}).

\begin{lemma}\label{lem:final}
 $(\g_{2n+1})$ is a $\tmod$-regular sequence.
\end{lemma}
\begin{proof}
 We will prove the equivalent statement: $(\g_{2n+1}^{-1})$ is a $\tmod$-regular subsequence, see  \Cref{foot:regular}. 
 We observe that $\g_{2n+1}^{-1}$ has the following form:
 \[
  \g_{2n+1}^{-1} = \omega_{2n+1}^{-1} \delta_{2n+1}^{-1}
  = \underbrace{\alpha_1^{-1}\beta_1^{-1} \cdots \alpha_{n+1}^{-1}}_{\omega_{2n+1}^{-1}} \underbrace{\beta^{(2n+1)}_1 \alpha^{(2n+1)}_1 \cdots \alpha^{(2n+1)}_q}_{\delta_{2n+1}^{-1}}
 \]
 Therefore,
 \[
  \g_{2n+1}^{-1} B \subset \omega_{2n+1}^{-1} B.
 \]
 Since $(\omega^{-1}_{2n+1})_{n\in\N}$ is type A alternating,
 by \Cref{lem:int_singleton}, $\omega_{2n+1}^{-1} B$ converges to a point $\tau_+ \in\Ft$.
 Therefore, by the above inclusion, we also obtain 
 \[
  \g_{2n+1}^{-1} B \ra \tau_+,
  \quad \text{as } n\ra\infty.
 \]
 By \Cref{prop:regulairty}, $(\g_{2n+1}^{-1})$ is $\tmod$-regular subsequence.
\end{proof}

Hence, the sequence $(\g_n')$ contains a $\tmod$-regular subsequence, namely $(\g_{2n+1})$.
This is a contradiction!
\end{proof}

\section{Boundary embedding}\label{sec:boundary}

In this section, we work under the hypothesis of \Cref{thm:main}, i.e., we assume that $\G_A$ and $\G_B$ are $\tmod$-Anosov subgroups admitting a ping-pong pair $(A,B)$ in $\Ft$ such that $A$ and $B$ are antipodal to each other.
By \Cref{lem:Lt}, we know that $\Lt(\G_A) \subset A$, and $\Lt(\G_B) \subset B$.
Moreover, since $\tmod$-Anosov subgroups of $G$ are $\tmod$-regular (see \Cref{sec:discrete_groups}), by \Cref{thm:regularity}, $\G \coloneqq \inp*{\G_A}{\G_B} \cong \G_A * \G_B$ is a $\tmod$-regular subgroup of $G$.

Anosov subgroups are intrinsically word-hyperbolic (see \Cref{def:AE}). Hence, the free product $\G_A * \G_B$ is also word-hyperbolic.
We fix a finite generating set $S = S_A \cup S_B$ for $\G_A * \G_B$, where $S_A$ and $S_B$ generate $\G_A$ and $\G_B$, respectively.
Using this choice of $S$ we define the Cayley graph of $\G_A * \G_B$.
We equip $\G_A * \G_B$ with the corresponding word-metric. 

The goal of this section is to construct a $\G$-equivariant antipodal embedding
\[
 \xi : \vb\G \ra A\cup B\subset \Ft,
\]
such that
\begin{equation}\label{eqn:restrict}
   \xi_A = \xi\vert_{\vb\G_A} \quad\text{and} \quad
\xi_B =  \xi\vert_{\vb\G_B},
\end{equation}
where the homeomorphisms $\xi_A :  \vb\G_A \ra \Lt(\G_A)$ and $\xi_B : \vb\G_B \ra \Lt(\G_B)$ are the asymptotic embeddings for $\G_A$ and $\G_B$, respectively.

We first recall that there exists a $\G$-invariant decomposition
\[
 \vb \G = \G (\vb \G_A \sqcup \vb \G_B) \,\sqcup\, \vb T
\]
where $T$ is the Bass--Serre tree of the free product $\G_A * \G_B$, see e.g. \cite{MS}.
For the notational convenience, we introduce the following notation: 
\[
\vi \coloneqq \G (\vb \G_A \sqcup \vb \G_B) \quad\text{and} \quad \vii  \coloneqq \vb T.
\]
We will first define the map $\xi$ separately on these two disjoint subsets $\vi$ and $\vii$. Then we will  piece them together and check that $\xi$ is an equivariant antipodal embedding.

\subsection{Construction of $\xi$}\label{sec:Construction} 
We define a map 
\begin{equation}\label{eqn:def_vi_map}
 \xi : \vi \to \Ft
\end{equation}
 as follows:
Let $\e\in\vi$ be arbitrary.
There exists some $\g\in\G$ such that $\g^{-1} \e \in \vb \G_A \sqcup \vb \G_B$.
Suppose, for instance, that\footnote{the other possibility that $\g^{-1} \e \in \vb \G_B$ can be treated by relabelling} $\g^{-1} \e \in \vb \G_A$.
Define
\[
 \xi(\e) = \g \,\xi_A(\g^{-1}\e).
\]
To check well-definedness, we note that such a $\g$ is unique up to right multiplications by elements of $\G_A$.
Thus, if $\g_1\in \G$ is another choice such that $\g_1^{-1} \e \in \vi$, then $\alpha = \g^{-1}\g_1 \in \G_A$.
Hence,
$\g_1\, \xi_A(\g_1^{-1}\e) = \gamma\alpha\, \xi_A(\alpha^{-1}\gamma^{-1}\e) = \gamma\, \xi_A(\alpha\alpha^{-1}\gamma^{-1}\e) =\gamma\,\xi_A (\gamma^{-1} \e)$.
Moreover, by the construction, the map $\xi : \vi \ra \Ft$ is $\G$-equivariant and satisfies \eqref{eqn:restrict}.

\begin{proposition}\label{lem:image}
Let $\e\in\vi$, and let $(\g_n)$ be any sequence in $\G_A * \G_B$ such that $\g_n \raC \e$ in the compactified (by adding the Gromov boundary) Cayley graph of $\G_A * \G_B$.
Then, viewing $(\g_n)$ as a sequence in $G$, $\g_n \raf \xi(\e)$.

Moreover, the image $\xi(\vi)$  is contained in $A\cup B$. 
\end{proposition}

\begin{proof}Recall that by  \Cref{lem:Lt}, 
$\Lt(\G_A)\subset A$ and $\Lt(\G_B)\subset B$. 

For the first part, it is enough to assume that $\e\in \vb \G_A \sqcup \vb \G_B$.
 Suppose that $\e\in\vb\G_A$ (the other case when $\e\in\vb\G_B$ is treated similarly).
 Since $\g_n \raC \e \in \vb\G_A$, possibly after disregarding the first few terms, the 
 leftmost letter sequence $(\alpha_n)$ of $(\gamma_n)$ comes from $\G_A$ and $\alpha_n \raC \e$.
 Since $\G_A< G$ is $\tmod$-regular and $B$ is antipodal relative to $\Lt(\G_A)$, $\alpha_n(B) \to \{\xi(\e)\}$ in $\Ft$.
 Now,  \Cref{prop:regulairty} implies that $\g_n \raf \xi(\e)$; see \Cref{lem3.3case2} in the proof of \Cref{lem:bdd_length} for a similar argument.

 We prove the second part by induction on the relative lengths of elements $\gamma\in \G=\G_A * \G_B$. 
If, say, $\e\in \vb \G_A$ and 
$\gamma\in \G_A \cup \G_B \setminus \{1\}$, then either $\gamma(\e)\in \Lt(\Gamma_A)\subset A$ (when $\gamma\in \Gamma_A$) or $\gamma(\e)\in B$ (since $\gamma(A)\subset B$ in this case). Inductively, assume that for each element $\gamma\in\G$ of the relative length $n-1$ and of the reduced form $\alpha_1\beta_1\cdots$, then $\gamma(\e)\in A$, while if $\gamma$ is of the reduced form $\beta_1 \alpha_2\cdots$, then $\gamma(\e)\in B$. Consider an element $\gamma=\alpha_1 \beta_1 \cdots$ of the relative length $n$. Then, by the induction assumption, 
$$
e'= \beta_1\alpha_2\cdots(\e)\in B. 
$$ 
Then $\gamma(\e) =\alpha_1(e')\in A$ by the assumption of \Cref{thm:main}. Similarly, if $\gamma=\beta_1 \alpha_2\cdots$, then 
$\gamma(\e)\in B$.
\end{proof}

To define a map $\xi : \vii \ra \Ft$, we use the following result.

\begin{proposition}\label{prop:vii_map}
Let $\e\in\vii$, and let $(\g_n)$ be any sequence in $\G_A * \G_B$ such that $\g_n \raC \e$ in the compactified Cayley graph of $\G_A * \G_B$.
 Then, there exists $\tau_+\in \Ft$ depending only on $\e$ such that in $\Ft$, $\g_n \raf \tau_+$.
\end{proposition}

\begin{proof}
Consider a sequence $(\omega_n)$ in $\G_A * \G_B$ such that $\omega_n \raC \e$ and which is {\em alternating}, see \eqref{eqn:specialform}.
We remark that the alternating sequence lies  on the geodesic ray $\ell_\e$ in the Cayley graph of $\G_A*\G_B$ that emanates at the identity element and is asymptotic to $\e$.
Such an alternating sequence is {\em uniquely} determined by $\e\in\vii$.

\begin{figure}[h]
\centering
\begin{overpic}[scale=.5,tics=5]{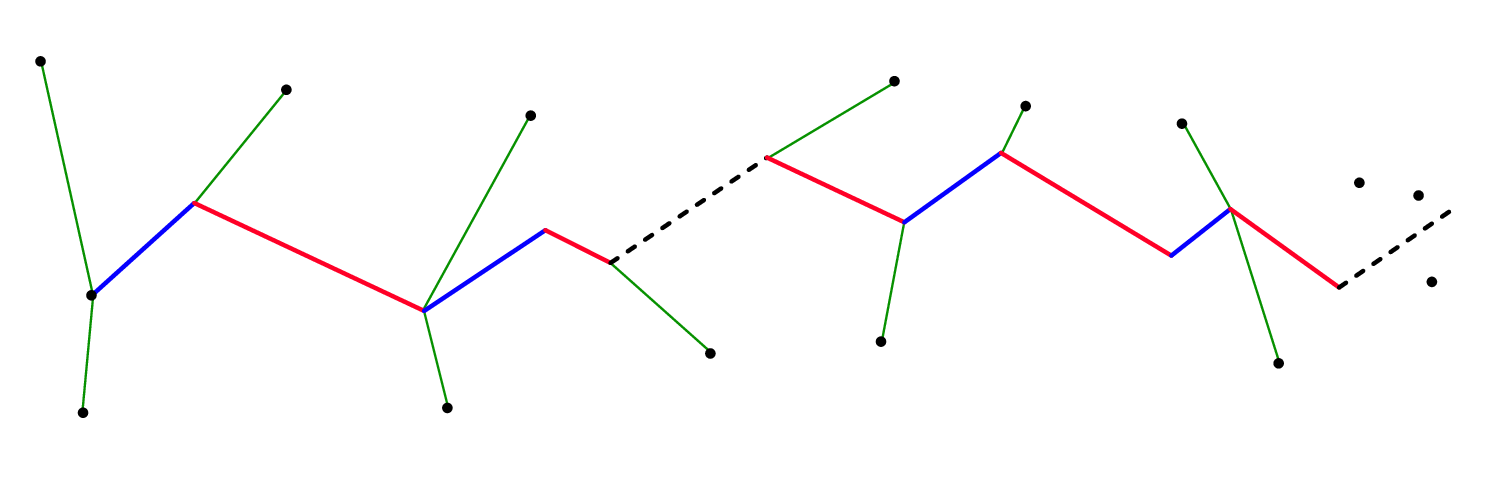}
\put(1,12){$1_\G$}
\put(97.5,18){\large$\e$}
\put(1.5,30){\small$\g_1$}\put(4.5,2.5){\small$\g_2$}\put(17.5,28){\small$\g_3$}\put(34,26.5){\small$\g_4$}\put(28.5,2.5){\small$\g_5$}\put(67,26.5){\small$\g_n$}\put(75,25.5){\small$\g_{n+1}$}\put(82,6){\small$\g_{n+2}$}
\put(6.5,17){\color{blue}\small$\alpha_1$}\put(32,12.5){\color{blue}\small$\alpha_2$}\put(62.5,17.5){\color{blue}\small$\alpha_{f(n)}$}
\put(20,16.5){\color{red}\small$\beta_1$}\put(37.5,17){\color{red}\small$\beta_2$}\put(71.5,19.5){\color{red}\small$\beta_{f(n)}$}
\put(13,23){\color{ForestGreen}\small$\delta_3$}\put(29,19.5){\color{ForestGreen}\small$\delta_4$}\put(64,23.5){\color{ForestGreen}\small$\delta_n$}\put(84.5,11){\color{ForestGreen}\small$\delta_{n+2}$}\put(80.5,21){\color{ForestGreen}\small$\delta_{n+1}$}
\end{overpic}
\caption{Fellow traveling ray}\label{fig:geodesic}
\end{figure}
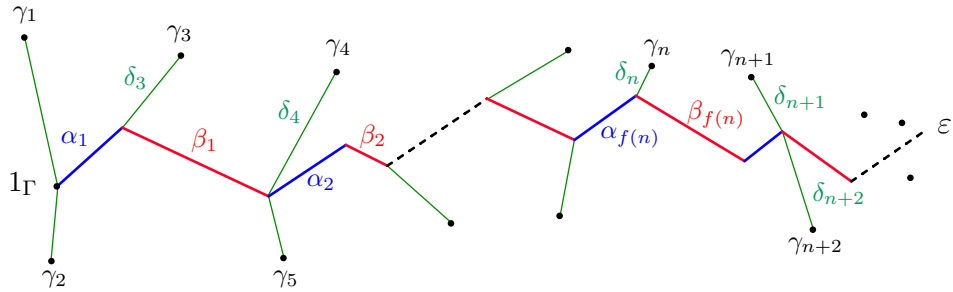

Let $(\g_n)$ be a sequence as in the statement of the proposition.
Since $\g_n \raC\e$, the sequence $(\g_n)$ ``fellow travels'' the geodesic ray $\ell_\e$. 
See \Cref{fig:geodesic}.
More precisely, there exists a sequence $(\delta_n)$ in $\G_A*\G_B$ such that the following hold: 
\begin{enumerate}[label=(\roman*)]\itemsep0em
\item There exists a nondecreasing function $f: \N \ra \N$ such that $\lim_{n\ra\infty}f(n) = \infty$ and $n_0\in\N$ such that, for all $n\ge n_0$,
 $\g_n = \omega_{f(n)} \delta_n$.
\item For $n\ge n_0$, if $\delta_n$ is nontrivial and the leftmost letter of the reduced form of $\delta_n$ is in $\G_A$ (resp. $\G_B$), then the rightmost letter of the reduced form of $\omega_{f(n)}$ is in $\G_B$ (resp. $\G_A$). 
\end{enumerate}

We further assume that the sequence $(\omega_n)$ is of type A in \eqref{eqn:specialform} (type B sequences  can be analyzed in a similar way).
So, for $n\ge n_0$, $\g_n$ has the following form
\[
 \g_n = \underbrace{\alpha_1 \beta_1 \cdots \beta_{f(n)-1}\alpha_{f(n)}}_{\omega_{f(n)}} \underbrace{\vphantom{\alpha_1 \beta_1 \dots \beta_{f(n)-1}\alpha_{f(n)}}\beta_1^{(n)} \alpha_1^{(n)} \cdots}_{\delta_n}.
\]
Next, we split the sequence $(\g_n)$ into two disjoint subsequences: The first (resp. second) subsequence $(\g_n')$ (resp. $(\g_n'')$) consists of all elements of $(\g_n)$ whose rightmost letters in their reduced forms are in $\G_A$ (resp. $\G_B$).
Let 
\begin{equation}\label{eqn:def_map_vii}
 \tau_+ \coloneqq  \bigcap_{i\in\N} \omega_n B = \lim_{n\ra\infty} \omega_n B
\end{equation}
(cf. \Cref{lem:int_singleton}).
Then, applying \Cref{lem:int_singleton}, we get
\[
 \g_n'(B) \raf \tau_+, \quad \g_n''(A) \raf \tau_+, \quad \text{ as }n\ra\infty,
\]
and then \Cref{prop:regulairty} implies $\g_n' \raf\tau_+$ and $\g_n'' \raf\tau_+$. Compare with the proof of \Cref{lem:final}.
Hence, $\g_n\ra\tau_+$.
\end{proof}

Using \Cref{prop:vii_map}, we define a map 
\begin{equation}\label{eqn:def_vii_map}
  \xi: \vii \ra \Ft, \quad
 \e\mapsto \tau_+,
\end{equation}
 where $\tau_+$ is the unique point in $\Ft$ corresponding to $\e$ obtained by the proposition.
This map is $\G$-equivariant: Let $\g\in \G$ and $\e\in\vii$.
If $(\omega_n)$ is an alternating sequence converging to $\e$, then the sequence
$(\g\omega_n)$ converges to $\g\e$.
By \Cref{prop:vii_map} and the definition of $\tau_+$ in \eqref{eqn:def_map_vii}, 
\[
 \xi(\g \e) = \lim_{n\ra\infty} (\g\omega_n) B \,=\, \g \left(\lim_{n\ra\infty} \omega_n B\right) = \g \, \xi(\e).
\]

\begin{corollary}\label{cor:bd_map}
 The piecewise-defined map $\xi : \vb\G \ra \Ft$ from \eqref{eqn:def_vi_map} and \eqref{eqn:def_vii_map} is a $\G$-equivariant map satisfying \eqref{eqn:restrict}.
\end{corollary}

\subsection{$\xi$ is antipodal}

In this subsection, we show that the map $\xi$ in \Cref{cor:bd_map} is antipodal: That is, if $\e,\hat\e \in\vb\G$ are distinct points, then $\xi(\e)$ is antipodal to $\xi(\hat\e)$ in $\Ft$.
In the proof,  we repeatedly use the fact that 
\begin{equation}\label{eqn:prsv_atpd}
 \text{ the action $G \acts \Ft$ preserves antipodality.}
\end{equation}

The discussion naturally divides into three mutually exclusive cases.

\setcounter{case}{0}
\begin{case}
 Suppose that $\e$ and $\hat\e$ both lie in $\vi$.
 By \eqref{eqn:prsv_atpd}, without loss of generality, we assume that $\e\in \vb\G_A \sqcup \vb\G_B$.
 If $\hat\e$ is also in $\vb\G_A \sqcup \vb\G_B$, then clearly $\xi(\e)$ is antipodal to $\xi(\hat\e)$ because $\xi_A \sqcup \xi_B : \vb\G_A \sqcup \vb\G_B \ra \Ft$ is an antipodal embedding.
 Thus, we assume that $\hat\e\not\in \vb\G_A \sqcup \vb\G_B$.
 Suppose further that $\e\in \vb\G_A$ (the other case $\e\in \vb\G_B$ can be analyzed similarly).
 Pick some $\g\in\G$ of {\em least} relative length such that  $\g\hat\e' = \hat\e$ for some $\hat\e' \in \vb\G_A \sqcup \vb\G_B$.

 Suppose that $\hat\e'\in \vb\G_A$. Then, by the assumption that $\g$ has the least relative length, the rightmost letter of the reduced form of $\g$ is in $\G_B\setminus\{1\}$.
If the leftmost letter of the reduced form of $\g$ is some element $\alpha\in\G_A \setminus \{1\}$, 
then we must have $\rell(\g) \ge 2$.
In this case, the second letter from the left of the reduced form of $\g$ is some element $\beta\in \G_B$, which is now the leftmost letter of the reduced form of $\alpha^{-1} \gamma$. 
Thus, $\alpha^{-1}\xi(\hat\e) = \alpha^{-1}\xi(\g\hat\e') = \alpha^{-1}\g \xi(\hat\e')  \in \beta A\subset B$ whereas $\alpha^{-1}\xi(\e)  \in\Lt(\G_A) \subset A$,
Hence, $\alpha^{-1}\xi(\hat\e)$ and $\alpha^{-1}\xi(\e)$ are antipodal, and consequently, by \eqref{eqn:prsv_atpd}, $\xi(\hat\e)$ and $\xi(\e)$ are antipodal.
Similarly, if the leftmost letter of the reduced form of $\g$ is in $\G_B \setminus \{1\}$, then we must have $\xi(\hat\e)\in B$, and since $\xi(\e)\in A$, $\xi(\hat\e)$ and $\xi(\e)$ are antipodal.

Now we assume that $\hat\e'\in \vb\G_B$.
Since $\g$ has the least relative length, the rightmost letter of the reduced form of $\g$ lies in $\G_A\setminus\{1\}$.
If the leftmost letter in the reduced form of $\g$ is some element $\alpha\in \G_A\setminus\{1\}$, then one may check that $\alpha^{-1}\xi(\hat\e) = \alpha^{-1}\gamma\xi(\hat\e')\in B$, which is antipodal to $\alpha^{-1}\xi(\e)  \in\Lt(\G_A) \subset A$.
Thus, by  \eqref{eqn:prsv_atpd}, $\xi(\hat\e)$ and $\xi(\e)$ are antipodal.
Else, if the leftmost letter in the reduced form of $\g$ is some element $\beta\in \G_B\setminus\{1\}$,
then $\rell(\g)\ge 2$ (since the rightmost letter lies in $\G_A\setminus\{1\}$).
In this case, by a similar reasoning as in the preceding paragraph, it follows that $\beta^{-1}\xi(\hat\e)\in A$ is antipodal to $\beta^{-1}\xi(\e)\in \beta^{-1} A \subset B$, which implies (again, by  \eqref{eqn:prsv_atpd}) that $\xi(\hat\e)$ and $\xi(\e)$ are antipodal.
\end{case}

\begin{case}\label{case:atpd_two}
 Suppose that $\e$ and $\hat\e$ both lie in $\vii$.
 Then $\e$ (resp. $\hat\e$) determines\footnote{See the beginning of the proof of \Cref{prop:vii_map}.} an alternating sequence $(\omega_n)$ (resp. $(\hat\omega_n)$) in $\G$.
 If the sequences $(\omega_n)$ and $(\hat\omega_n)$ are of different types (see \eqref{eqn:specialform}), then clearly,
 $\xi(\e)$ and $\xi(\hat\e)$ lie in different sets $A$ and $B$.
 Hence, $\xi(\e)$ and $\xi(\hat\e)$ are antipodal.
 
 Otherwise, the sequences $(\omega_n)$ and $(\hat\omega_n)$ have the same type, say type A.
 Let $N\in\N$ be the least number for which $\omega_N \ne \hat\omega_N$.
 Clearly, the sequences $(\omega_{N}^{-1}\omega_n)_{n\ge N+1}$ and $(\omega_N^{-1}\hat\omega_n)_{n\ge N+1}$ are still alternating, but of {\em different} types.
 Therefore, by the previous paragraph, $\omega_N^{-1}\xi(\e) = \xi(\omega_N^{-1}\e)$ and $\omega_N^{-1}\xi(\hat\e) = \xi(\omega_N^{-1}\hat\e)$
 are antipodal.
 Hence, by \eqref{eqn:prsv_atpd}, $\xi(\e)$ and $\xi(\hat\e)$ are antipodal.
\end{case}

\begin{case}
 Suppose that $\e \in\vi$ and $\hat\e\in\vii$.
 By \eqref{eqn:prsv_atpd}, without loss of generality, we assume that $\e\in \vb\G_A \sqcup \vb\G_B$.
 Suppose further that $\e\in \vb\G_A$; the other case $\e\in \vb\G_B$ can be analyzed by relabelling.
 Let $(\omega_n)$ be the alternating sequence in $\G$ determined by $\hat\e$.
 If $(\omega_n)$ is of type $B$, then $\xi(\hat\e) \in B$ and, hence, it is antipodal to $\xi(\e)$. Otherwise, if $(\omega_n)$ is of type $A$, then $\omega_1\in\G_A\setminus \{1\}$.
 In this case, by a similar argument as in \Cref{case:atpd_two}, $\omega^{-1}_1\xi(\e)$ is antipodal to $\omega^{-1}_1\xi(\hat\e)$ which, in conjunction with \eqref{eqn:prsv_atpd}, implies that $\xi(\e)$ is antipodal to $\xi(\hat\e)$.
\end{case}

Combining the above cases, we obtain the following:

\begin{proposition}\label{prop:antipodal}
 The map $\xi : \vb\G \ra \Ft$ in \Cref{cor:bd_map} is antipodal.
 
 In particular, $\xi$ is injective.
\end{proposition}

\subsection{$\xi$ is an embedding}\label{sec:embedding} 

Finally, we prove that the map $\xi : \vb\G \ra \Ft$ in \Cref{cor:bd_map} is an embedding whose image lies in $A\cup B$.
Since $\xi$ is injective (see \Cref{prop:antipodal}), domain of $\xi$ is a compact space, and the codomain is Hausdorff, it is enough to show that $\xi$ is continuous. It suffices to prove that for each point $\e\in \vb \Gamma$ and every sequence $\gamma_n\in \Gamma$ converging to $\e$, the sequence $(\gamma_n)$ in $G$ flag-converges to $\xi(\e)$. This, however, is the content of \Cref{lem:image} and \Cref{prop:vii_map}. 

In order to prove that $\xi(\vb \G)\subset A\cup B$, we observe that $\vi$ is dense in $\geo \G$ and that 
$\xi(\vi)\subset A\cup B$ (see the second part of \Cref{lem:image}). 
Thus, continuity of $\xi$ and the fact that $A\cup B$ is closed, implies that $\xi(\vb \G)\subset A\cup B$.

\section{Conclusion of the proof of Theorem \ref*{thm:main}}\label{sec:mainproof} 

Recall that a $\tmod$-regular subgroup $\G$ of $G$, which is word-hyperbolic as an abstract group, is $\tmod$-asymptotically embedded (see \Cref{def:AE}) if there exists an equivariant antipodal homeomorphism $f: \vb \Gamma\to \Lt$, the flag-limit set of $\G$ in $\Ft$, which extends continuously the orbit map \[o_x: \Gamma \to \Gamma x\subset X\] to the symmetric space $X$, where continuity is understood with respect to the topology of flag-convergence. The equivariant map $f$ is said to be the {\em asymptotic embedding} of $\G$. 

\begin{proposition}
Under the assumption of \Cref{thm:main}, the subgroup $\G$ of $G$ generated by $\G_A$ and $\G_B$ is $\tmod$-asymptotically embedded with the asymptotic embedding given by the map $\xi$ whose image lies in $A\cup B$. 
\end{proposition}
\proof We already know that the subgroup $\G< G$ is $\tmod$-regular. See \Cref{thm:regularity}.

The discussion in  \Cref{sec:embedding} shows that  
$\xi: \vb \G \to \Ft$ is a continuous antipodal map whose image is contained in the flag-limit set of $\G$ in $\Ft$ and which continuously extends the orbit map $o_x$. The fact that the map $\xi$ is onto $\Lt$ follows from the fact that every sequence in $\G$ subconverges to a point 
$\e\in \vb \G$. Let $(\gamma_n)$ denote that convergent subsequence. Then, as we noted above, $o_x(\gamma_n)$ converges   to $\xi(\e)$, i.e. the flag-limit of the sequence $\gamma_n\in \Gamma< G$ is $\xi(\e)\in \xi(\vb \G)$. \qed  

\medskip
This concludes the proof of the \Cref{thm:main}.

\section{Consequences of Theorem \ref*{thm:main}}\label{sec:consequences}

Below we present some  applications of the main theorem of the paper. 

\subsection{Schottky subgroups}

A subgroup $\G$ of $G$ is said to be a  {\em $\tmod$-Schottky} subgroup if it is free and $\tmod$-Anosov. 
Starting with a subset of pairwise antipodal points $\{a_\pm,b_\pm \}$ in $\Ft$, let $\alpha,\beta\in G$ be {\em axial} isometries of $X$ which preserves some $\tmod$-regular geodesic lines $L_{a_\pm}, L_{b_\pm} \subset X$, respectively, such that $L_{a_\pm}$ (resp. $L_{b_\pm}$) is forward/backward asymptotic to $a_\pm$ (resp. $b_\pm$).
It was shown in \cite[Section 7.6]{morse} that, after passing to sufficiently large  powers of $\alpha$ and $\beta$, 
the subgroup $\inp*{\alpha^m}{\beta^n}$, $m,n\gg 1$, is a $\tmod$-Schottky subgroup of $G$, and is naturally isomorphic to the two-generator free group. Below, 
as a consequence of \Cref{thm:main}, we show that under suitable conditions, $\inp*{\alpha}{\beta^n}$, $n\gg 1$, is a $\tmod$-Schottky subgroup of $G$, i.e., under these conditions, we  need to take a large power of only {\em one} of the generators.
In particular, this construction produces a large family of  2-generated $\tmod$-Schottky subgroups such that one of the generators 
can be assumed to have 
arbitrarily small {\em translation length} in $X$.
We remark that the constructions of such groups do not follow from the results of
\cite{morse} or \cite{MR4002289}.

Let $\{a_\pm,b_\pm\}$, $\alpha$, and $\beta$ be as above. We will assume that there exists a  
one-parameter closed subgroup $H_\alpha$  of $G$ containing the element $\alpha$ and such that all elements of $H_\alpha$ preserve the geodesic  $L_{a_\pm}$. (This holds, for instance, if $\alpha$ is a {\em transvection} along $L_{a_\pm}$.) 

\begin{lemma}\label{lem:generic_iso}
 All but finitely many elements of $\hat\alpha\in H_\alpha$ satisfy the following: The subset $\{b_+,b_-\}$ is antipodal relative to $\hat\alpha\{b_+,b_-\}$ in $\Ft$.
\end{lemma}

\begin{proof}
 Let $f: \R \ra H_\alpha< G$ be the isomorphism normalized by $f(1) = \alpha$. 
 Under this identification, we have an analytic action $\R \acts \Ft^2$ given by $r\cdot (\tau_+,\tau_-) = (f(r)\tau_+,f(r)\tau_-)$.
 In particular, if $C \subset\Ft^2$ is an analytic subset and $(\tau_+,\tau_-)$ is an arbitrary point, then $E_C(\tau_+,\tau_-) \coloneqq \{ r\in\R \mid r\cdot (\tau_+,\tau_-) \in C \}$ is an analytic subset of $\R$.
 
 Recall the notion of exceptional subvariety $E(\tau) \subset \Ft$ for $\tau$ from \Cref{sec:basics}.
 Setting 
 \[
 C = (E(b_\pm) \times \Ft) \cup (\Ft \times E(b_\pm)),
 \]
 we get that $E_C(b_+,b_-)$ is an analytic subset of $\R$.
 Furthermore, $(\tau_+,\tau_-)\in \Ft^2$ 
 lies in the complement of $C$ if and only if $\{\tau_+,\tau_-\}$ is antipodal relative to $\{b_+,b_-\}$.
 Since 
 $$
 \lim_{r\ra\pm\infty} r\cdot (b_+,b_-) = (a_\pm,a_\pm)
 $$
  and $\{a_+,a_-\}$ is antipodal relative to $\{b_+,b_-\}$, $E_C(b_+,b_-)$ must be a compact subset of $\R$.
 Hence, $E_C(b_+,b_-)$ is finite.
\end{proof}
 
Now we return to the construction of $\tmod$-Schottky subgroups of $G$.
 Suppose  that the axial isometries $\alpha$ and $\beta$ of $X$ are chosen so that:
\begin{enumerate}\itemsep0em
\item As before, the forward/backward asymptotic points of $L_\alpha$ and $L_\beta$  in $\Ft$ are pairwise antipodal points $a_{\pm}$ and $b_\pm$, where $L_\alpha$ (resp. $L_\beta$) is a $\tmod$-regular line in $X$ preserved by $\alpha$ (resp. $\beta$). 
 
\item  For all $k\in\Z \setminus \{0\}$, $\alpha^k \{b_+,b_-\}$ is antipodal relative to $\{b_+,b_-\}$. 
\end{enumerate}

Note that \Cref{lem:generic_iso} guarantees that all but at most countably many elements of $H_\alpha$ satisfy the second condition above.
In particular, one can choose $\alpha$ to have an arbitrarily small translation length in $X$.

We have:

\begin{proposition}[Schottky subgroups]
 For all large $n$, the subgroup $\inp*{\alpha}{\beta^n}$ of $G$ is $\tmod$-Anosov, and is naturally isomorphic to the free group of rank two.
\end{proposition}

\begin{proof}
 We observe that $A_1 = \{ \alpha^k b_\pm \mid k\in \Z\setminus \{0\}\} \cup \{ a_\pm \}$ and $B_1 = \{b_\pm\}$ are compact sets, which are antipodal relative to each other (by the condition 2 above). Moreover, by our choice of $A_1$, $\alpha^k(B)\subset A$ for all $k\ne 0$. 
 There exist subsets $A$ and $B$ of $\Ft$ such that $A$ and $B$ are antipodal relative to each other, and $\Int A \supset A_1$, $\Int B \supset B_1$, see \cite[Lemma 4.24]{MR4002289}.
 Moreover, for all large $k\in \N$, $\alpha^{\pm k} B \subset \Int A$.
 Making $B$ smaller we can assure that $\alpha^k B \subset \Int A$, for all $k\in\Z \setminus \{0\}$.
 
 On the other hand,
 $$
 \lim_{k\to\pm\infty} \alpha^k(A)=\{b_\pm\}. 
 $$
 Hence, for all sufficiently large $n\in\N$, $\beta^{\pm n} A \subset \Int B$. 
 Applying the \Cref{thm:main} to $\G_A = \< \alpha\>$, $\G_B = \<\beta^n\>$, $n\gg1$, and the subsets $A$ and $B$ as above, we finish the proof.
\end{proof}

\subsection{A general version of the \Cref*{thm:main} involving several Anosov subgroups}
Our second application of the main theorem is its generalization to the case of several Anosov subgroups:

\begin{corollary}\label{cor:several_anosov}
Suppose that the subgroups $\Gamma_1,\dots,\Gamma_n$ of $G$ are $\tmod$-Anosov, $A_1,\dots,A_n\subset \Ft$ are pairwise antipodal compact subsets with nonempty interiors such that for every nontrivial element $\gamma_i\in \Gamma_i$,
\begin{equation}
 \label{eqn:ping-pong-several}
\gamma_i(A_1 \cup \dots \cup A_n \setminus A_i) \subset \Int A_i, 
\end{equation}
$i=1,\dots,n$. Then:

 \begin{enumerate}[label=(\roman*)]\itemsep0em
 \item The subgroup $\G< G$ generated by $\G_1,\dots,\G_n$ is naturally isomorphic to the  free product $\G_1 * \G_2 * \cdots *\G_n$. 
 
 \item The subgroup $\G< G$ is $\tmod$-Anosov. 
 
 \item The $\tmod$-limit set of $\G$ is contained in $A_1\cup \dots \cup A_n$.
 \end{enumerate}
\end{corollary}
\proof The proof is by induction on $n$. Set $\Gamma_A:=\< \G_1,\dots,\G_{k-1}\>$, $\Gamma_B=\Gamma_k$,
$$
A:= A_1\cup \dots\cup A_{k-1}, \quad B:= A_k, \quad k\le n.  
$$
We assume that $\Gamma_A$ is $\tmod$-Anosov, naturally isomorphic to $\G_1 * \G_2 * \dots *\G_{k-1}$,  and that each nontrivial element $\alpha\in \Gamma_A$ satisfies 
$$
\alpha(B)\subset \Int A
$$
(it is clear that every nontrivial element $\beta\in \G_B$ satisfies $\beta(A)\subset \Int B)$.  Then \Cref{thm:main} implies that 
$\G^k:=\<\G_A, \G_B\>$ is naturally isomorphic to the free product   $\G_A * \G_B\cong \G_1 * \G_2 * \dots *\G_k$ and is $\tmod$-Anosov. In order to prove that each nontrivial element $\gamma\in \G^k$ satisfies $\gamma(A_{j})\subset A\cup B$, $j=k+1,\dots,n$, one argues by induction on the relative length of $\gamma\in \G_A * \G_B$, similarly to the proof of  \Cref{lem:image}. We leave the details to the reader. 
 \qed 
 
\begin{remark}
 We show that the main result (Theorem 1.3) of our paper with Bernhard Leeb, \cite{MR4002289}, follows from the \Cref{cor:several_anosov}: If $\G_1,\dots,\G_n$ are {\em pairwise antipodal},\footnote{That is, the $\tmod$-limit sets of $\G_i$ and $\G_j$, for all $i\ne j$, are antipodal relative to each other.} residually finite, $\tmod$-Anosov subgroups, then consider compact neighborhoods $A_1,\dots, A_n$ of the limit sets of $\G_1,\dots,\G_n$, respectively, such that $A_i$ and $A_j$, for all $i\ne j$, are pairwise antipodal (cf. \cite[Lemma 4.24]{MR4002289}).
 For each $i$, the subset of $\G_i$ not satisfying \eqref{eqn:ping-pong-several} is  finite. This can be seen as follows: Suppose, to the contrary, that there exists an infinite sequence $(\g_k)$ of distinct elements in $\G_i$ such that, for all $k\in\N$, $\gamma_k(A_1 \cup \dots \cup A_n \setminus A_i) \not\subset \Int A_i$.
 In particular, the Hausdorff distance (with respect to a metric on $\Ft$ compatible with the manifold topology) between $\gamma_k(A_1 \cup \dots \cup A_n \setminus A_i)$ and any point in $\Lt(\G_i)$ is uniformly bounded below by a positive number. However, passing to a subsequence, we can ensure that $\g_k \raf \tau_+ \in\Lt(\G_i)\subset\Int A_i$.
 Hence, $\gamma_k(A_1 \cup \dots \cup A_n \setminus A_i) \ra \tau_+$, a contradiction.
 
 Thus, by the residual finiteness assumption, for each $i$, we pass to a finite index subgroup $\G_i'<\G_i$ such that \eqref{eqn:ping-pong-several} is satisfied for the collections $\G_1',\dots,\G_n'$ and $A_1,\dots, A_n$.
 Applying \Cref{cor:several_anosov} to these collections, we obtain that the subgroup $\inp*{\G_1',\dots}{\G_n'}$ of $G$ is $\tmod$-Anosov and naturally isomorphic to the abstract free product $\G_1'*\dots*\G_n'$, which is  the conclusion of \cite[Theorem 1.3]{MR4002289}.
\end{remark}

\subsection{Another example of a Schottky group}

In this section we give an example (claimed in the introduction) of a Schottky group failing the {\em uniform contraction} property but otherwise satisfying the assumptions of the Combination Theorem.  

 \begin{example}\label{ex1}
 Consider $G= \PSL(2, \R)$, the group of orientation-preserving isometries of the hyperbolic plane (in the unit disk model), acting on the unit circle $S^1$. Take two pairwise disjoint closed round disks $B_\pm, A_\pm$ whose boundary circles have the same radii and meet $S^1$ orthogonally. Take $\alpha, \beta\in G$ such that $\alpha$ 
 sends $A_+$ to the complement (in the extended complex plane) of $A_-$ and $\beta$  sends $B_+$ to the complement of $B_-$. See Figure \ref{fig:example}.  The elements $\alpha, \beta$ can be chosen so that their restrictions to the boundary circles of $A_-, B_+$ respectively are isometries of the angular metrics on the respective circles. For instance, one can take $\alpha$ (respectively, $\beta$) the composition of the inversion in $\partial A_-$ (respectively, $\partial B_-$) 
 and a reflection in a line passing through the center of the unit disk.  Then the subgroups $\Gamma_A=\langle \alpha\rangle, \Gamma_B=\langle \beta \rangle$ and  the subsets $A=A_-\cup A_+$, $B=B_-\cup B_+$ satisfy  the assumptions of the Combination Theorem in the extended complex plane. It remains to modify $A, B$ to new subsets $A', B'$, still satisfying the assumptions of the Combination Theorem (for the $G$-action on $S^1$), so that $\beta$ is no longer is a strict expansion on $B'$. We let $c$ denote the (compact) arc on $S^1$ connecting a point of $\partial A_-$ to a point of $\partial B_-$ and, otherwise, disjoint from $A\cup B$. (Such an arc always exists if we are willing to replace $\alpha$ with $\alpha^{-1}$.) See again Figure \ref{fig:example}. Now set $A':= (A_+\cap S^1) \cup (A_-\cap S^1) \cup c$. We also replace $B_-\cap S^1$ with a slightly smaller compact subarc $b_-\subset B_-\cap S^1$ and $B_+\cap S^1$ with a slightly larger compact subarc $b_+\subset B_+\cap S^1$, so that 
 $$
 \beta^{\pm 1}(A')\subset b_+ \cup b_-.  
 $$
 Lastly, set $B'=b_- \cup b_+$. Then the quadruple $\Gamma_A, \Gamma_B, A', B'$ satisfies the assumptions of the Combination Theorem, but $\beta$ is not expanding on $B$, specifically, the expansion fails at the intersection point $x$ of $c$ and the circle $\partial B_-$. 
 \end{example}

\begin{figure}[h]
\centering
\begin{overpic}[scale=.6,tics=5]{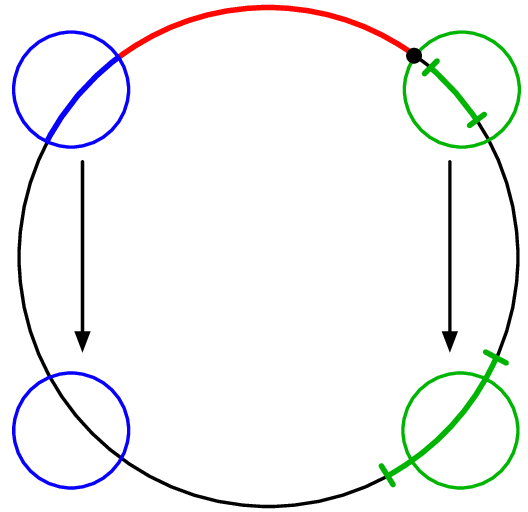}
\put(-3,90){\color{blue}$A_-$}
\put(-1,0){\color{blue}$A_+$}
\put(97,0){\color{ForestGreen}$B_+$}
\put(97,90){\color{ForestGreen}$B_-$}
\put(49,88){\color{red}$c$}
\put(18,46){$\alpha$}
\put(78,46){$\beta$}
\put(76,89){$x$}
\end{overpic}
\caption{}\label{fig:example}
\end{figure}

\bibliographystyle{alpha}

\vspace{.3in}

\noindent
S.D.

Department of Mathematics,
   Yale University, 
   219 Prospect St, New Haven, CT 06511

\texttt{subhadip.dey@yale.edu}

\medskip
\noindent
M.K.

{Department of Mathematics,
   University of California, Davis,
   One Shields Ave, Davis, CA 95616}

\texttt{kapovich@ucdavis.edu}

\end{document}